\documentclass[11pt]{article}
\usepackage{amsmath,amssymb,amsthm}

\newtheorem{theorem}{Theorem}[section]
\newtheorem{lemma}[theorem]{Lemma}
\newtheorem{proposition}[theorem]{Proposition}
\newtheorem{corollary}[theorem]{Corollary}

\theoremstyle{plain}

\theoremstyle{definition}
\newtheorem{definition}[theorem]{Definition}

\numberwithin{equation}{section}

\renewcommand{\labelenumi}{\textup{(\theenumi)}}

\renewcommand{\phi}{\varphi}

\newcommand{\Homeo}{\operatorname{Homeo}}
\newcommand{\id}{\operatorname{id}}
\newcommand{\Ker}{\operatorname{Ker}}
\newcommand{\Ad}{\operatorname{Ad}}

\newcommand{\N}{\mathbb{N}}
\newcommand{\T}{\mathbb{T}}
\newcommand{\Z}{\mathbb{Z}}
\newcommand{\Zp}{{\mathbb{Z}}_+}

\title{Strongly continuous orbit equivalence of \\
one-sided topological Markov shifts}
\author{Kengo Matsumoto \\
Department of Mathematics \\
Joetsu University of Education \\
Joetsu, 943-8512, Japan
}
\date{}

\begin{document}
\maketitle

\def\det{{{\operatorname{det}}}}

\begin{abstract}
We will introduce a notion of strongly continuous orbit equivalence 
in one-sided topological Markov shifts.
Strongly continuous orbit equivalence
yields
a topological conjugacy 
between their two-sided topological
Markov shifts
$(\bar{X}_A, \bar{\sigma}_A)$
and 
$(\bar{X}_B, \bar{\sigma}_B)$.
We prove that 
one-sided topological Markov shifts
$(X_A, \sigma_A)$
and 
$(X_B, \sigma_B)$
are strongly continuous orbit equivalent
if and only if 
there exists an isomorphism bewteen 
the Cuntz-Krieger algebras
${\mathcal{O}}_A$ and 
${\mathcal{O}}_B$ 
preserving their maximal commutative $C^*$-subalgebras
$C(X_A)$ and $C(X_B)$
and giving  cocycle conjugate gauge actions.
An example of one-sided topological
Markov shifts which are
strongly continuous orbit equivalent 
but not one-sided topologically conjugate
is presented.
\end{abstract}




\def\OA{{{\mathcal{O}}_A}}
\def\OB{{{\mathcal{O}}_B}}
\def\FA{{{\mathcal{F}}_A}}
\def\FB{{{\mathcal{F}}_B}}
\def\DA{{{\mathcal{D}}_A}}
\def\DB{{{\mathcal{D}}_B}}
\def\HA{{{\frak H}_A}}
\def\HB{{{\frak H}_B}}
\def\Ext{{{\operatorname{Ext}}}}
\def\Max{{{\operatorname{Max}}}}
\def\Per{{{\operatorname{Per}}}}
\def\PerB{{{\operatorname{PerB}}}}
\def\Homeo{{{\operatorname{Homeo}}}}
\def\HSA{{H_{\sigma_A}(X_A)}}
\def\Out{{{\operatorname{Out}}}}
\def\Aut{{{\operatorname{Aut}}}}
\def\Ad{{{\operatorname{Ad}}}}
\def\Inn{{{\operatorname{Inn}}}}
\def\det{{{\operatorname{det}}}}
\def\exp{{{\operatorname{exp}}}}
\def\cobdy{{{\operatorname{cobdy}}}}
\def\Ker{{{\operatorname{Ker}}}}
\def\ind{{{\operatorname{ind}}}}
\def\id{{{\operatorname{id}}}}
\def\supp{{{\operatorname{supp}}}}
\def\COE{{{\operatorname{COE}}}}
\def\SCOE{{{\operatorname{SCOE}}}}


\section{Introduction}

Let $A=[A(i,j)]_{i,j=1}^N$ 
be an $N\times N$ matrix with entries in $\{0,1\}$,
where $1< N \in {\Bbb N}$.
Throughout the paper, 
we assume that 
$A$ is irreducible and satisfies condition (I) in the sense of Cuntz--Krieger \cite{CK}.
We denote by 
$X_A$ the shift space 
\begin{equation}
X_A = \{ (x_n )_{n \in \N} \in \{1,\dots,N \}^{\N}
\mid
A(x_n,x_{n+1}) =1 \text{ for all } n \in {\N}
\} \label{eq:Markovshift}
\end{equation}
of the right one-sided topological Markov shift for $A$.
It is a compact Hausdorff space in natural  product topology
on $\{1,\dots,N\}^{\Bbb N}$.
The shift transformation $\sigma_A$ on $X_A$ defined by 
$\sigma_{A}((x_n)_{n \in \Bbb N})=(x_{n+1} )_{n \in \Bbb N}$
is a continuous surjective map on $X_A$.
The topological dynamical system 
$(X_A, \sigma_A)$ is called the (right) one-sided topological Markov shift for $A$.
The two-sided topological Markov shift
written $(\bar{X}_A, \bar{\sigma}_A)$
is similarly defined by gathering two-sided
sequences 
$(x_n)_{n \in \Z}$ 
instead of one-sided sequences 
$(x_n )_{n \in \N}$ in \eqref{eq:Markovshift}.   
In \cite{CK}, J. Cuntz and W. Krieger have introduced a $C^*$-algebra  from topological Markov shift $(X_A,\sigma_A)$.
It is called the Cuntz--Krieger algebra written $\OA$.
They have proved that if one-sided topological Markov shifts $(X_A,\sigma_A)$ and   
$(X_B,\sigma_B)$ are topologically conjugate,
then the Cuntz--Krieger algebras 
$\OA$ and $\OB$ with their gauge actions are conjugate. 
They have also proved that 
if two-sided topological Markov shifts 
$(\bar{X}_A,\bar{\sigma}_A)$ and   
$(\bar{X}_B,\bar{\sigma}_B)$ 
are topologically conjugate,
then the stabilized Cuntz--Krieger algebras 
$\OA\otimes{\mathcal{K}}(H)$ and $\OB\otimes{\mathcal{K}}(H)$ 
with their stabilized gauge actions are conjugate.
We note that one-sided topological conjugacy of topological Markov shifts
yields two-sided topological conjugacy. 
The author in \cite{MaPacific}
has introduced 
the notion of continuous orbit equivalence of one-sided topological Markov shifts.
It is a dynamical equivalence relation in one-sided topological Markov shifts
inspired by studies of orbit equivalences in Cantor minimal systems by 
 Giordano--Putnam--Skau (cf. \cite{GPS}, \cite{GPS2}), 
 Giordano--Matui--Putnam--Skau (cf. \cite{GMPS}).
 It is a weaker than one-sided topological conjugacy
 and
 gives rise to isomorphic Cuntz--Krieger algebras (\cite{MaPacific}).
Let $A$ and $B$ be irreducible square matrices with entries in $\{0,1 \}$.
One-sided topological Markov shifts
$(X_A, \sigma_A)$ and $(X_B,\sigma_B)$ 
are said to be continuously orbit equivalent 
if there exists a homeomorphism
$h: X_A \rightarrow X_B$ 
such that
\begin{align}
\sigma_B^{k_1(x)} (h(\sigma_A(x))) 
& = \sigma_B^{l_1(x)}(h(x))
\quad 
\text{ for} \quad 
x \in X_A,  \label{eq:orbiteq1x} \\
\sigma_A^{k_2(y)} (h^{-1}(\sigma_B(y))) 
& = \sigma_A^{l_2(y)}(h^{-1}(y))
\quad 
\text{ for } \quad 
y \in X_B \label{eq:orbiteq2y}
\end{align}
for some continuous functions 
$k_1,l_1 \in C(X_A, \Zp),\, 
 k_2,l_2 \in C(X_B, \Zp). 
$
Let
$G_A$ denote the \'etale groupoid
for $(X_A,\sigma_A)$ 
whose reduced groupoid $C^*$-algebra 
$C^*_{r}(G_A)$ is isomorphic to the Cuntz--Krieger algebra 
 $\mathcal{O}_A$
(cf. \cite{MatuiPLMS}, \cite{MatuiPre2012}, \cite{Renault}).
Denote by $\DA$ the canonical maximal abelian $C^*$-subalgebra 
of $\OA$ realized as the commutative $C^*$-algebra of continuous functions on the unit space
$G^{(0)}_A$ of $G_A$.
The algebra $\DA$ is canonically isomorphic to the $C^*$-algebra
$C(X_A)$ of continuous functions on the shift space $X_A$.
H. Matui has studied  continuous orbit equivalence from the view point of groupoids
(\cite{MatuiPLMS}, \cite{MatuiPre2012}).

In \cite{MMKyoto},
we have obtained the following classification results 
of continuous orbit equivalence of one-sided topological Markov shifts.
\begin{theorem}[\cite{MMKyoto}, cf. \cite{MaPacific}, \cite{MaPAMS}, 
\cite{MatuiPLMS}, \cite{MatuiPre2012}]
The following four assertions are equivalent:
\begin{enumerate}
\renewcommand{\theenumi}{\roman{enumi}}
\renewcommand{\labelenumi}{\textup{(\theenumi)}}
\item
 $(X_A, \sigma_A)$ and $(X_B,\sigma_B)$ 
are continuously orbit equivalent.
\item The \'etale groupoids $G_A$ and $G_B$ are isomorphic. 
\item There exists an isomorphism $\Psi:\mathcal{O}_A\to\mathcal{O}_B$ 
such that $\Psi(\mathcal{D}_A)=\mathcal{D}_B$. 
\item $\mathcal{O}_A$ and $\mathcal{O}_B$ 
are isomorphic and $\det(\id-A)=\det(\id-B)$. 
\end{enumerate}
\end{theorem}
Let $A$ be an irreducible square matrix with entries in $\{0,1\}$.
The ordered cohomology group $(\bar{H}^A,\bar{H}^A_+)$
is defined by the quotient group of the ordered abelian group
$C(\bar{X}_A,{\mathbb{Z}})$ of all ${\mathbb{Z}}$-valued continuous functions 
on $\bar{X}_A$ quoted by the subgroup
$\{ \xi - \xi \circ \bar{\sigma}_A \mid \xi \in C(\bar{X}_A,{\mathbb{Z}}) \}$.
The positive cone $\bar{H}^A_+$ consists of the classes of nonnegative functions in $C(\bar{X}_A,{\mathbb{Z}})$ 
(cf. \cite{BH}, \cite{Po}).
We similarly define the ordered cohomology group
$(H^A,H^A_+)$
for one-sided topological Markov shift $(X_A,\sigma_A)$.
The latter ordered group  
$(H^A,H^A_+)$ is naturally isomorphic to the former one  
$(\bar{H}^A,\bar{H}^A_+)$ (\cite[Lemma 3.1]{MMKyoto}).
The ordered group
$(H^A,H^A_+)$ 
is also isomorphic to the first cohomology group
$H^1(G_A,\Z)$ 
of the groupoid
$G_A$ (\cite[Proposition 3.4]{MMKyoto}).
In \cite{BH}, Boyle--Handelman have proved that 
the ordered cohomology group
$(\bar{H}^A,\bar{H}^A_+)$ 
is a complete invariant for  flow equivalence  of 
two-sided topological Markov shift
$(\bar{X}_A, \bar{\sigma}_A)$.

In the first part of this paper, 
we will introduce a notion of continuous orbit map
from  
$(X_A, \sigma_A)$ to $(X_B,\sigma_B)$.
A local homeomorphism 
$h:X_A \rightarrow X_B$ 
is said to be {\it continuous orbit map\/}
if
there exist continuous functions
$k_1,\, l_1:X_A \rightarrow \Zp$
such that
\begin{equation}
\sigma_B^{k_1(x)} (h(\sigma_A(x))) = \sigma_B^{l_1(x)}(h(x))
\quad \text{ for}
\quad x \in X_A. \label{eq:orbiteqx}
\end{equation}
It yields a morphism in the continuous orbit equivalence classes 
of one-sided topological Markov shifts.
For $f \in C(X_B,\Z )$, define
\begin{equation}
\Psi_{h}(f)(x)
= \sum_{i=0}^{l_1(x)-1} f(\sigma_B^i(h(x))) 
- \sum_{j=0}^{k_1(x)-1} f(\sigma_B^j(h(\sigma_A(x))))
 \quad 
\text{ for }
\quad x \in X_A. \label{eq:Psihfx}
\end{equation}
It is easy to see that 
$\Psi_h(f) \in C(X_A,\Z)$.
Thus $\Psi_h: C(X_B,\Z) \rightarrow C(X_A,\Z)$
gives rise to a homomorphism of abelian groups
and induces a homomorphism
from $H^B$ to $H^A$.
We will then show that 
the objects of  continuous orbit equivalence classes
of one-sided topological Markov shifts
with the morphisms of continuous orbit maps
form a category (Proposition \ref{prop:COEcategory}).  
We have
\begin{theorem}[{Theorem \ref{thm:contravariant}}]
The correspondence $\Psi$ yields a contravariant functor from 
the category of continuous orbit 
equivalence classes $[(X_A,\sigma_A)]$
of one-sided topological Markov shifts to 
that of ordered abelian groups $(H^A, H^A_+)$.
\end{theorem}
The class $[1_A]\in H^A$
of the constant function $1_A(x) =1, x \in X_A$
is an order unit of the ordered group $(H^A, H^A_+)$.
Let $h:X_A \rightarrow X_B$ be a continuous orbit map
giving rise to a 
continuous orbit equivalence 
between 
$(X_A,\sigma_A)$ and $(X_B,\sigma_B)$.
It is a key ingredient 
in the papers \cite{MMKyoto}, \cite{MMPre2014}
that the class
$[\Psi_h(1_B)]$ in $H^A$ of 
$\Psi_h(1_B)\in C(X_A,\Z)$   
belongs to the positive cone $H^A_+$. 
In the second part of this paper,
we will  introduce a notion of strongly continuous orbit equivalence
between 
$(X_A,\sigma_A)$ and $(X_B,\sigma_B)$,
which is defiend by the condition
that $[\Psi_h(1_B)] =[1_A]$ in $H^A$.
It has been proved in \cite{MMPre2014} that 
under the condtion 
that $[\Psi_h(1_B)] =[1_A]$ in $H^A$,
their zeta functions coincide, that is,
$\det(\id -tA) = \det(\id - t B)$.
Hence strongly continuous orbit equivalence
preserves the structure of  periodic points
of their two-sided topological Markov shifts.
We will know  that
strongly  continuous orbit equivalence 
in one-sided topological Markov shifts is a subequivalence relation
in continuous orbit equivalence,
so that 
the objects of strongly  continuous orbit equivalence classes
of one-sided topological Markov shifts
with the morphisms 
 of strongly continuous orbit maps
form a category
(Proposition \ref{prop:SCOEcategory}).
Continuous orbit equivalence of one-sided topological Markov shifts 
does not necessarily give rise to topological conjugacy of their  
two-sided topological Markov shifts.
We however see the following theorem:
\begin{theorem}[{Theorem \ref{thm:twoconjugacy} and {Corollary \ref{cor:crossedXA}}}]
Suppose that
$(X_A, \sigma_A)$ and $(X_B,\sigma_B)$
are strongly  continuous orbit equivalent.
Then their two-sided topological Markov shifts
$(\bar{X}_A, \bar{\sigma}_A)$ and 
$(\bar{X}_B, \bar{\sigma}_B)$
are topologically conjugate.
Hence
their $C^*$-crossed products are isomorphic:
\begin{equation*}
C(\bar{X}_A) \times_{\bar{\sigma}_A^*} \Z 
\cong
C(\bar{X}_B) \times_{\bar{\sigma}_B^*} \Z. 
\end{equation*}
 \end{theorem}

Let us denote by
$\rho^A$ the gauge action on $\OA$.
In general, continuous orbit equivalence
does not necessarily yield 
the cocycle conjugacy of the gauge actions on the Cuntz--Krieger algebras.
We have the following result which is a generalization of \cite[2.17 Proposition]{CK}.

\begin{theorem}[{Theorem \ref{thm:TFAE}}]
The following two assertions are equivalent.
\begin{enumerate}
\renewcommand{\theenumi}{\roman{enumi}}
\renewcommand{\labelenumi}{\textup{(\theenumi)}}
\item
One-sided topological Markov shifts
$(X_A, \sigma_A)$ and $(X_B,\sigma_B)$
are strongly continuous orbit equivalent.
\item
There exist
a unitary one-cocycle $v_t \in U(\OB), t \in \T$
for the gauge action $\rho^B$ on $\OB$
and
 an isomorphism
$\Phi:\OA \rightarrow \OB$ such that 
\begin{equation*}
\Phi(\DA) = \DB
\quad 
\text{  and }
\quad
\Phi \circ \rho^A_t = \Ad(v_t)\circ \rho^B_t \circ \Phi, 
\qquad t \in {\mathbb{T}}.
\end{equation*}
\end{enumerate}
\end{theorem}
Hence if
$(X_A, \sigma_A)$ and $(X_B,\sigma_B)$
are strongly continuous orbit equivalent,
then  the dual actions of the gauge actions on the Cuntz--Krieger algebras 
are isomorphic (Corollary \ref{cor:crossedgauge}):
\begin{equation*}
(\OA{\times}_{\rho^A}{\mathbb{T}}, \hat{\rho}^A, \Z)
 \cong
(\OB{\times}_{\rho^B}{\mathbb{T}}, \hat{\rho}^B, \Z).
\end{equation*}

One-sided topological conjugacy yields a strongly continuous 
orbit equivalence.
We will finally present an example of 
one-sided topological Markov shifts 
$(X_A, \sigma_A)$ and 
$(X_B, \sigma_B)$
which are
strongly continuous orbit equivalent
but not topologically conjugate. 
Let $A$ and $B$ be the following matrices:
\begin{equation}
A=
\begin{bmatrix}
1 & 1 \\
1 & 1
\end{bmatrix},
\qquad
B=
\begin{bmatrix}
1 & 1& 0 \\
1 & 0& 1 \\
1 & 0& 1 
\end{bmatrix}. \label{eq:exmatrices}
\end{equation}
They are both irreducible and satisfy condition (I).
We will show the following theorem.
\begin{theorem} [{Theorem \ref{thm:ExSCOE}}]
The one-sided topological Markov shifts 
$(X_A, \sigma_A)$ and 
$(X_B, \sigma_B)$
defined by the matrices \eqref{eq:exmatrices}
are strongly continuous orbit equivalent,
but not topologically conjugate.
\end{theorem}

Throughout the paper,
we will use the following notations.
The set of positive integers
and
the set of nonnegative integers
are denoted 
by $\N$
and
by $\Zp$
respectively.
A word $\mu = \mu_1 \cdots \mu_k$ for $\mu_i \in \{1,\dots,N\}$
is said to be admissible for $X_A$ 
if $\mu$ appears in somewhere in an element $x$ in $X_A$.
The length of $\mu$ is $k$, 
which is denoted by $|\mu|$.
 We denote by 
$B_k(X_A)$ the set of all admissible words of length $k$.
We set 
$B_*(X_A) = \cup_{k=0}^\infty B_k(X_A)$ 
where $B_0(X_A)$ denotes  the empty word $\emptyset$.
Denote by $U_\mu$ the cylinder set 
$\{ (x_n )_{n \in \Bbb N} \in X_A 
\mid x_1 =\mu_1,\dots, x_k = \mu_k \}$
for $\mu=\mu_1\cdots\mu_k \in B_k(X_A)$.  
For $x = (x_n )_{n \in \Bbb N} \in X_A$ and 
$k,l \in {\mathbb{N}}$ with $k \le l$,
we set 
\begin{equation*}
x_{[k,l]}  = x_k x_{k+1}\cdots x_l \in B_{l-k+1}(X_A),\qquad
x_{[k,\infty)}  = (x_k,x_{k+1},\dots ) \in X_A.
\end{equation*}
We denote by 
$C(X_A,\Zp)$
the set of 
$\Zp$-valued continuous functions on $X_A$.
A point $x\in X_A$ is said to be eventually periodic
if $\sigma_A^r(x) = \sigma_A^s(x)$
for some $r,s \in \Zp$ with $r \ne s$.

\section{Continuous orbit maps}
\begin{definition}\label{defn:orbitmap}
Let $(X_A, \sigma_A)$ and $(X_B,\sigma_B)$ 
be one-sided  topological Markov shifts.
A local homeomorphism 
$h:X_A \rightarrow X_B$ 
is said to be {\it continuous orbit map\/}
if
there exist continuous functions
$k_1,\, l_1:X_A \rightarrow \Zp$
such that
\begin{equation}
\sigma_B^{k_1(x)} (h(\sigma_A(x))) = \sigma_B^{l_1(x)}(h(x))
\quad \text{ for}
\quad x \in X_A. \label{eq:orbiteqx}
\end{equation}
It is denoted  
by 
$h:(X_A,\sigma_A) \rightarrow (X_B,\sigma_B)$.
If a continuous orbit map
$h:(X_A,\sigma_A) \rightarrow (X_B,\sigma_B)$
 is a homeomorphism
such that its inverse
$h^{-1}:(X_B,\sigma_B) \rightarrow (X_A,\sigma_A)$
is also  a continuous orbit map,
it is called a {\it continuous orbit homeomorphism\/}.
\end{definition}
Hence 
 $(X_A, \sigma_A)$ and $(X_B,\sigma_B)$
 are continuously orbit equivalent if and only if 
 there exists a 
continuous orbit homeomorphism 
$h:(X_A,\sigma_A) \rightarrow (X_B,\sigma_B)$.

For a continuous orbit map
$h:(X_A,\sigma_A) \rightarrow (X_B,\sigma_B)$
with continuous functions
$k_1,\, l_1:X_A \rightarrow \Zp$
satisfying \eqref{eq:orbiteqx},
we put for $n \in {\mathbb{N}}$, 
\begin{equation*}
k_1^n(x) = \sum_{i=0}^{n-1}k_1(\sigma_A^i(x)),\qquad
l_1^n(x) = \sum_{i=0}^{n-1}l_1(\sigma_A^i(x)) \qquad 
\text{ for }
\quad
x \in X_A. 
\end{equation*}
We note that the following identities hold.
\begin{lemma}[cf. {\cite[Lemma 3.1]{MMPre2014}}]
For $n,m \in \Zp$, we have
\begin{align} 
k_1^{n+m}(x) & = k_1^n(x) + k_1^m(\sigma_A^n(x)) 
\quad 
\text{ for }
\quad
x \in X_A, \\
l_1^{n+m}(x) & = l_1^n(x) + l_1^m(\sigma_A^n(x))
\quad
\text{ for }
 \quad x \in X_A
\end{align}
and
\begin{equation}
\sigma_B^{k_1^n(x)}(h(\sigma_A^n(x)))
= \sigma_B^{l_1^n(x)}(h(x)) 
\quad 
\text{ for }
\quad 
x \in X_A. \label{eq:norbiteqx} 
\end{equation}
\end{lemma}

\begin{lemma}\label{lem:gh}
Let $A, B, C$ be irreducible matrices with entries in $\{0,1\}$
satisfying condition (I).
Let
$h:(X_A,\sigma_A) \rightarrow (X_B,\sigma_B)$
and
$g:(X_B,\sigma_B) \rightarrow (X_C,\sigma_C)$
be continuous orbit maps such that
there exist continuous functions
$k_1,\, l_1:X_A \rightarrow \Zp$
and
$k_2,\, l_2:X_B \rightarrow \Zp$
satisfying
\begin{align}
\sigma_B^{k_1(x)} (h(\sigma_A(x))) & = \sigma_B^{l_1(x)}(h(x))
\quad \text{ for}
\quad x \in X_A, \label{eq:sBkh}\\ 
\sigma_C^{k_2(y)} (g(\sigma_B(y))) & = \sigma_C^{l_2(y)}(g(y))
\quad \text{ for}
\quad y \in X_B. \label{eq:sCkg}
\end{align}
Put
\begin{align}
k_{3}(x) & = k_2^{l_1(x)}(h(x)) + l_2^{k_1(x)}(h(\sigma_A(x)))  
\quad 
\text{ for }
\quad x \in X_A, \label{eq:k3} \\
l_{3}(x) & = l_2^{l_1(x)}(h(x)) + k_2^{k_1(x)}(h(\sigma_A(x)))
\quad 
\text{ for }
\quad x \in X_A. \label{eq:l3} 
\end{align}
Then we have
\begin{equation}
\sigma_C^{k_{3}(x)} (g\circ h(\sigma_A(x)))  
= \sigma_C^{l_{3}(x)}(g\circ h(x))
\quad \text{ for}
\quad x \in X_A.
\end{equation}
Hence 
$g \circ h :X_A \rightarrow X_C$
gives rise to a continuous orbit map.
\end{lemma}
\begin{proof}
Take an arbitrary element  $x \in X_A$.
For $n \in \N$ and $y \in X_B$, we have
by \eqref{eq:norbiteqx} 
\begin{equation}
\sigma_C^{k_2^n(y)} (g(\sigma_B^n(y))) = \sigma_C^{l_2^n(y)}(g(y)). \label{eq:CBn}
\end{equation}
Apply \eqref{eq:CBn} for $n=l_1(x), y =h(x)$, one has
\begin{equation*}
\sigma_C^{k_2^{l_1(x)}(h(x))} (g(\sigma_B^{l_1(x)}( h(x)))) 
= \sigma_C^{l_2^{l_1(x)}(h(x))}(g(h(x))). \label{eq:CBn1}
\end{equation*}
Apply \eqref{eq:CBn} for $n=k_1(x), y =h(\sigma_A(x))$, one has
\begin{equation*}
\sigma_C^{k_2^{k_1(x)}(h(\sigma_A(x)))} 
(g(\sigma_B^{k_1(x)}( h(\sigma_A(x))))) 
= \sigma_C^{l_2^{k_1(x)}(h(\sigma_A(x)))}(g(h(\sigma_A(x)))). \label{eq:CBn2}
\end{equation*}
Put
$n= l_1(x),m = k_1(x)$.
By \eqref{eq:sBkh},
we have
\begin{align*}
\sigma_C^{k_2^n(h(x)) + l_2^m(h(\sigma_A(x)))}(g\circ h(\sigma_A(x)))
& =\sigma_C^{k_2^n(h(x))}(
\sigma_C^{k_2^m(h(\sigma_A(x)))}(g(\sigma_B^m( h(\sigma_A(x)))))) \\
& =\sigma_C^{k_2^n(h(x))}(
\sigma_C^{k_2^m(h(\sigma_A(x)))}(g(\sigma_B^n( h(x))))) \\
& =\sigma_C^{k_g^m(h(\sigma_A(x)))}(
\sigma_C^{l_2^n(h(x))}(g(h(x))) \\
& =\sigma_C^{ k_2^m(h(\sigma_A(x))) + l_2^n(h(x))}(g(h(x))).
\end{align*}
\end{proof}
More generally 
we have the following formula.
The proof is routine.
\begin{lemma}
Keep the above situation.
For $p \in \N$, we have
\begin{equation*}
\sigma_C^{k_2^n(h(x)) + l_2^m(h(\sigma_A^p(x)))}(g\circ h(\sigma_A^p(x)))
=\sigma_C^{ k_2^m(h(\sigma_A^p(x))) + l_2^n(h(x))}(g(h(x)))
\end{equation*}
where
$n=l_1^p(x), m= k_1^p(x)$.
\end{lemma}
We have the following proposition.
\begin{proposition}
Let $A, B, C, D$ be irreducible matrices with entries in $\{0,1\}$
satisfying condition (I).
Let
$h:(X_A,\sigma_A) \rightarrow (X_B,\sigma_B)$
and
$h':(X_C,\sigma_C) \rightarrow (X_D,\sigma_D)$
be
continuously orbit homeomorphisms.
If there exists a continuous orbit map
$g:(X_A,\sigma_A) \rightarrow (X_C,\sigma_C)$,
then
the map
$h'\circ g \circ h^{-1}: X_B \rightarrow X_D$ 
becomes a continuous orbit map.
Hence continuous orbit maps form morphisms
of the continuous orbit equivalence classes of 
one-sided topological Markov shifts.
\end{proposition}
Therefore we have
\begin{proposition}\label{prop:COEcategory}
The objects of  continuous orbit equivalence classes
of one-sided topological Markov shifts
with the morphisms of continuous orbit maps
form a category.  
\end{proposition}

\section{Cohomology groups}
For a one-sided topological Markov shift
$(X_A,\sigma_A)$,
we denote by $\cobdy(\sigma_A)$
the subgroup
$\{\xi-\xi\circ\sigma_A\mid\xi\in C(X_A,{\mathbb{Z}})\}$
of $C(X_A,\Z)$, 
and set the quotient group
\begin{equation*}
H^A=C(X_A,{\mathbb{Z}})/\{\xi-\xi\circ\sigma_A\mid\xi\in C(X_A,{\mathbb{Z}})\}.
\end{equation*}
In this section,
we will construct a contravariant functor
$\Psi$
from the category of continuous orbit equivalence classes 
of one-sided topological Markov shifts
to the category of abelian groups.
Let
$h:(X_A,\sigma_A) \rightarrow (X_B,\sigma_B)$ be
 a continuous orbit map
with continuous functions 
$k_1, l_1: X_A \rightarrow \Z$
satisfying \eqref{eq:orbiteqx}.
 \begin{lemma}
The function $c_1(x) = l_1(x) - k_1(x)$ for $x \in X_A$
does not depend on the choice of the continuous functions $k_1, l_1$
satisfing \eqref{eq:orbiteqx}.
\end{lemma}
\begin{proof}
Let 
$k'_1, l'_1 \in C(X_A,\Z)$
be another continuous functions for $h$ satisfying
\begin{equation}
\sigma_B^{k'_1(x)} (h(\sigma_A(x))) = \sigma_B^{l'_1(x)}(h(x))
\quad \text{ for}
\quad x \in X_A. \label{eq:primeorbiteqx}
\end{equation}
Since $k_1, k'_1$ are both continuous,
there exists $K \in \N$ such that 
$k_1(x), k'_1(x) \le K$ for all $x \in X_A$. 
Put
$c'_1(x) = l'_1(x) - k'_1(x)$ 
so that
we have
\begin{equation*}
 \sigma_B^{c_1(x)+K}(h(x))
 = \sigma_B^{K}(h(\sigma_A(x))) 
 = \sigma_B^{c'_1(x)+K}(h(x)) \quad 
 \text{ for } x \in X_A. 
\end{equation*}
Suppose that 
$c_1(x_0) \ne c'_1(x_0)$ for some $x_0 \in X_A$.
There exists a clopen neighborhood $U$ of $x_0$
such that 
$$
c_1(x) \ne c'_1(x) \quad \text{ for all } x \in U.
$$
Now $h$ is a local homeomorphism,
one may take a clopen neighborhood 
$V \subset U$ of $x_0$ such that
$h: V \rightarrow h(V)$ is a homeomorphism.
As  
$
c_1(x) +K \ne c'_1(x) +K
$
for all $x \in V$,
$h(x)$ are eventually periodic points
for all $x \in V$,
which is a contradiction to
the fact that the set of non eventualy periodic points of $h(V)$ is dense in $h(V)$.
\end{proof}
We call the above function $c_1$
the cocycle function of $h$. 
For $f \in C(X_B,\Z )$, define
\begin{equation}
\Psi_{h}(f)(x)
= \sum_{i=0}^{l_1(x)-1} f(\sigma_B^i(h(x))) 
- \sum_{j=0}^{k_1(x)-1} f(\sigma_B^j(h(\sigma_A(x)))),
 \quad 
 x \in X_A. \label{eq:Psihfx}
\end{equation}
It is easy to see that 
$\Psi_h(f) \in C(X_A,\Z)$.
Thus $\Psi_h: C(X_B,\Z) \rightarrow C(X_A,\Z)$
gives rise to a homomorphism of abelian groups.

\begin{lemma}
$\Psi_h: C(X_B,\Z) \rightarrow C(X_A,\Z)$
does not depend on the choice of the functions 
$k_1, l_1$ satisfying \eqref{eq:orbiteqx}.
\end{lemma}
\begin{proof}
The proof is similar to the proof of
\cite[Lemma 4.2]{MMPre2014}.
\end{proof}

\noindent
{\bf Examples.}

{\bf 1.}
Let $h:X_A \rightarrow X_B$ be a topological conjugacy
as one-sided subshifts.
Then $h$ is a continuous orbit map such that 
$\Psi_{h}(f) = f \circ h$ for $f \in C(X_B,\Z)$.

{\bf 2.}
For $A=B$,
the shift map
$\sigma_A: X_A \rightarrow X_A$
is a continuous orbit map on $X_A$
such that 
$\Psi_{\sigma_A}(f) = f\circ  \sigma_A$ for $f \in C(X_A,\Z)$.

The equalities in the following lemma are 
basic in our further discussions.
The proof is similar to the proof of \cite[Lemma 4.3]{MMPre2014}.
\begin{lemma} \label{lem:Am}
Let
$h: (X_A,\sigma_A) \rightarrow (X_B,\sigma_B)$
be a continuous orbit map 
with continuous functions 
$k_1, l_1$ satisfying \eqref{eq:orbiteqx}.
For $f \in C(X_B,{\mathbb{Z}})$,
$x \in X_A$ and $m=1,2,\dots,$ 
the following equalities hold:
\begin{align*}
& \sum_{i=0}^{m-1} \{
\sum_{i'=0}^{l_1 (\sigma_B^i(x))-1} f(\sigma_B^{i'}(h(\sigma_A^i(x))))
-
\sum_{j'=0}^{k_1 (\sigma_A^i(x))-1} f(\sigma_B^{j'}(h(\sigma_A^{i}(x)))) \}\\
=
&
\sum_{i'=0}^{l_1^m (x)-1} f(\sigma_B^{i'}(h(x)))
-
\sum_{j'=0}^{k_1^m (x)-1} f(\sigma_B^{j'}(h(\sigma_A^{m}(x)))).
\end{align*}
\end{lemma}
Let
$h: (X_A,\sigma_A) \rightarrow (X_B,\sigma_B)$
and
$g: (X_B,\sigma_B) \rightarrow (X_C,\sigma_C)$
be continuous orbit maps 
with continuous functions 
$k_1, l_1 \in C(X_A,\Z)$
and
$k_2, l_2 \in C(X_B,\Z)$
satisfying 
\eqref{eq:sBkh} and \eqref{eq:sCkg},
respectively.
We write the 
continuous orbit map 
$g\circ h:  X_A \rightarrow  X_C$
as $gh$.
We will prove the following proposition
\begin{proposition}\label{prop:compo} 
$\Psi_h \circ \Psi_g = \Psi_{g h}$.
\end{proposition}
To prove the proposition, 
we provide some notations and a lemma. 
Let
$
k_3, l_3:X_A \rightarrow \Zp
$
be the continuous functions defined by
\eqref{eq:k3}, \eqref{eq:l3}, respectively.
By Lemma \ref{lem:gh}, we have for $f \in C(X_C,\Z)$ 
\begin{equation}
\Psi_{gh}(f)(x)
= \sum_{i=0}^{l_3(x)-1} f(\sigma_C^i(gh(x))) 
- \sum_{j=0}^{k_3(x)-1} f(\sigma_C^j(gh(\sigma_A(x)))),
 \quad 
 x \in X_A. \label{eq:Psighfx}
\end{equation}
Keep the above situations.
By Lemma \ref{lem:Am},
we have
\begin{lemma} \label{lem:Cm}
For $f \in C(X_C,{\mathbb{Z}})$,
$x \in X_A$ and $m=1,2,\dots,$ 
the following equalities hold:
\begin{align*}
& \sum_{i=0}^{m-1} \{
\sum_{i'=0}^{l_2 (\sigma_B^i(h(x)))-1} f(\sigma_C^{i'}(g(\sigma_B^i(h(x)))))
-
\sum_{j'=0}^{k_2 (\sigma_B^i(h(x)))-1} f(\sigma_C^{j'}(g(\sigma_B^{i}(h(x))))) \}\\
=
&
\sum_{i'=0}^{l_2^m (h(x))-1} f(\sigma_C^{i'}(gh(x)))
-
\sum_{j'=0}^{k_2^m (h(x))-1} f(\sigma_C^{j'}(g(\sigma_B^{m}(h(x))))).
\end{align*}
\end{lemma}
Hence we have
\begin{corollary}\label{cor:Psi}
For $f \in C(X_C,{\mathbb{Z}})$ and $x \in X_A$, we have 
\begin{enumerate}
\renewcommand{\theenumi}{\roman{enumi}}
\renewcommand{\labelenumi}{\textup{(\theenumi)}}
\item
\begin{align*}
& \sum_{i=0}^{l_1(x)-1} \{
\sum_{i'=0}^{l_2 (\sigma_B^i(h(x)))-1} 
f(\sigma_C^{i'}(g(\sigma_B^i(h(x)))))
-
\sum_{j'=0}^{k_2 (\sigma_B^i(h(x)))-1} 
f(\sigma_C^{j'}(g(\sigma_B^{i}(h(x))))) \}\\
=
&
\sum_{i'=0}^{l_2^{l_1(x)} (h(x))-1} 
f(\sigma_C^{i'}(gh(x)))
-
\sum_{j'=0}^{k_2^{l_1(x)}(h(x))-1} 
f(\sigma_C^{j'}(g(\sigma_B^{l_1(x)}(h(x)))))
\end{align*}
so that 
\begin{align*}
& \sum_{i=0}^{l_1(x)-1} 
\Psi_g(f)(\sigma_B^i(h(x)) )\\
=
&
\sum_{i'=0}^{l_2^{l_1(x)} (h(x))-1} 
f(\sigma_C^{i'}(gh(x)))
-
\sum_{j'=0}^{k_2^{l_1(x)}(h(x))-1}  
f(\sigma_C^{j'}(g(\sigma_B^{l_1(x)}(h(x))))).
\end{align*}
\item
\begin{align*}
& \sum_{i=0}^{k_1(x)-1} \{
\sum_{i'=0}^{l_2 (\sigma_B^i(h(\sigma_A(x))))-1} 
f(\sigma_C^{i'}(g(\sigma_B^i(h(\sigma_A(x))))) )
-
\sum_{j'=0}^{k_2 (\sigma_B^i(h(\sigma_A(x)))-1} 
f(\sigma_C^{j'}(g(\sigma_B^i(h(\sigma_A(x)))))) \} \\
=
&
\sum_{i'=0}^{l_2^{k_1(x)} (h(\sigma_A(x)))-1} 
f(\sigma_C^{i'}(gh(\sigma_A(x))))
-
\sum_{j'=0}^{k_2^{k_1(x)} (h(\sigma_A(x)))-1}
 f(\sigma_C^{j'}(g(\sigma_B^{k_1(x)}(h(\sigma_A(x))))))
\end{align*}
so that 
\begin{align*}
& \sum_{j=0}^{k_1(x)-1} 
\Psi_g(f)(\sigma_B^j(h(\sigma_A(x)))) \\
=
&
\sum_{i'=0}^{l_2^{k_1(x)} (h(\sigma_A(x)))-1} 
f(\sigma_C^{i'}(gh(\sigma_A(x))))
-
\sum_{j'=0}^{k_2^{k_1(x)}(h(\sigma_A(x)))-1}  
f(\sigma_C^{j'}(g(\sigma_B^{k_1(x)}(h(\sigma_A(x)))))).
\end{align*}
\end{enumerate}
\end{corollary}

\medskip

{\it{(Proof of Proposition \ref{prop:compo})}}

 By the previous corollary,
we have 
\begin{align*}
  &\Psi_{h}(\Psi_g(f))(x) \\
=&  \sum_{i=0}^{l_1(x)-1} \Psi_g(f)(\sigma_B^i(h(x)) )
- \sum_{j=0}^{k_1(x)-1} \Psi_g(f)(\sigma_B^j(h(\sigma_A(x))) ) \\
=
&
\{ 
\sum_{i'=0}^{l_2^{l_1(x)} (h(x))-1} f(\sigma_C^{i'}(gh(x)))
-
\sum_{j'=0}^{k_2^{l_1(x)}(h(x))-1}  
f(\sigma_C^{j'}(g(\sigma_B^{l_1(x)}(h(x)))))
\} 
\\
-&
\{
\sum_{i'=0}^{l_2^{k_1(x)} (h(\sigma_A(x)))-1} 
f(\sigma_C^{i'}(gh(\sigma_A(x))))
-
\sum_{j'=0}^{k_2^{k_1(x)}(h(\sigma_A(x)))-1}  
f(\sigma_C^{j'}(g(\sigma_B^{k_1(x)}(h(\sigma_A(x))))) \}. 
\end{align*}
By \eqref{eq:l3}, 
the first $\{ \, \cdot \, \}$ above  goes to
\begin{align}
& \sum_{i'=0}^{l_3(x)-1} f(\sigma_C^{i'}(gh(x))) 
\label{eq:circ1A} \\
-
& \{ \sum_{i'=l_2^{l_1(x)}(h(x))}^{l_3(x)-1} 
f(\sigma_C^{i'}(gh(x))) 
+
 \sum_{j'=0}^{k_2^{l_1(x)}(h(x))-1}  
f(\sigma_C^{j'}(g(\sigma_B^{l_1(x)}(h(x))))) \}.
 \label{eq:circ1C}
\end{align}
By \eqref{eq:k3},
the second $\{ \, \cdot \, \}$ above goes to
\begin{align}
& \sum_{i'=0}^{k_3(x)-1} 
f(\sigma_C^{i'}(gh(\sigma_A(x)))) 
\label{eq:circ2A} \\
-
& \{
\sum_{i'=l_2^{k_1(x)}(h(\sigma_A(x)))}^{k_3(x)-1} 
f(\sigma_C^{i'}(gh(\sigma_A(x)))) 
+
\sum_{j'=0}^{k_2^{k_1(x)}(h(\sigma_A(x)))-1}  
f(\sigma_C^{j'}(g(\sigma_B^{k_1(x)}(h(\sigma_A(x)))))) \}.
\label{eq:circ2C}
\end{align}
We thus see
\begin{equation*}
\Psi_{h}(\Psi_g(f))(x)
 = \{ \eqref{eq:circ1A}-\eqref{eq:circ1C}\}
  - \{ \eqref{eq:circ2A} -\eqref{eq:circ2C}\}. 
\end{equation*}
Since
$
\sigma_C^{l_2^{l_1(x)}(h(x))}(gh(x)) 
=\sigma_C^{k_2^{l_1(x)}(h(x))}(g(\sigma_B^{l_1(x)}(h(x)))),
$
we  have
\begin{align*}
& \eqref{eq:circ1C} \\
=
& \sum_{j'=0}^{k_2^{k_1(x)}(h(\sigma_A(x)))-1} 
f(\sigma_C^{j'}(\sigma_C^{l_2^{l_1(x)}(h(x))}(gh(x)))) 
+
 \sum_{j'=0}^{k_2^{l_1(x)}(h(x))-1}  
f(\sigma_C^{j'}(g(\sigma_B^{l_1(x)}(h(x))))) \\
=
& \sum_{j'=0}^{k_2^{k_1(x)}(h(\sigma_A(x)))-1} 
f(\sigma_C^{j'}(
\sigma_C^{k_2^{l_1(x)}(h(x))}(g(\sigma_B^{l_1(x)}(h(x))))) \\ 
&+
 \sum_{j'=0}^{k_2^{l_1(x)}(h(x))-1}  
f(\sigma_C^{j'}(g(\sigma_B^{l_1(x)}(h(x))))) \\
=
&
 \sum_{j'=0}^{k_2^{l_1(x)}(h(x))+ k_2^{k_1(x)}(h(\sigma_A(x)))-1}  
f(\sigma_C^{j'}(g(\sigma_B^{l_1(x)}(h(x))))). 
\end{align*}
Since
$\sigma_C^{l_2^{k_1(x)}(h(\sigma_A(x)))}(gh(\sigma_A(x))) 
=
\sigma_C^{k_2^{k_1(x)}(h(\sigma_A(x)))}(g(\sigma_B^{k_1(x)}(h(\sigma_A(x))))),
$
we  have
\begin{align*}
& \eqref{eq:circ2C} \\
=
& \sum_{j'=0}^{ k_2^{l_1(x)}(h(x))-1} 
f(\sigma_C^{j'}(\sigma_C^{l_2^{k_1(x)}(h(\sigma_A(x)))}(gh(\sigma_A(x))))) 
+
\sum_{j'=0}^{k_2^{k_1(x)}(h(\sigma_A(x)))-1}  
f(\sigma_C^{j'}(g(\sigma_B^{k_1(x)}(h(\sigma_A(x))))) \\
=
& \sum_{j'=0}^{ k_2^{l_1(x)}(h(x))-1} 
f(\sigma_C^{j'}(
\sigma_C^{k_2^{k_1(x)}(h(\sigma_A(x)))}(g(\sigma_B^{k_1(x)}(h(\sigma_A(x)))
) \\
&+
\sum_{j'=0}^{k_2^{k_1(x)}(h(\sigma_A(x)))-1}  
f(\sigma_C^{j'}(g(\sigma_B^{k_1(x)}(h(\sigma_A(x))))) \\
=
&
\sum_{j'=0}^{k_2^{k_1(x)}(h(\sigma_A(x))) +k_2^{l_1(x)}(h(x))-1}  
f(\sigma_C^{j'}(g(\sigma_B^{k_1(x)}(h(\sigma_A(x))))). 
\end{align*}
As
$\sigma_C^{j'}(g(\sigma_B^{k_1(x)}(h(\sigma_A(x))))
= \sigma_C^{j'}(g(\sigma_B^{l_1(x)}(h(x)))), 
$
we have
$\eqref{eq:circ1C} 
=
\eqref{eq:circ2C} 
$
so that
\begin{equation*}
\Psi_{h}(\Psi_g(f))(x)
= 
\eqref{eq:circ1A}-\eqref{eq:circ2A}
= \Psi_{gh}(f)(x).
\end{equation*}
\qed

\begin{corollary}[{\cite[Proposition 4.5]{MMPre2014}}]
Let $h:(X_A,\sigma_A)\rightarrow (X_B,\sigma_B)$
be a continuous orbit homeomorphism.
Then we have
$\Psi_h \circ \Psi_{h^{-1}}= \id_{C(X_A,\Z)}$
and
$\Psi_{h^{-1}} \circ \Psi_h= \id_{C(X_B,\Z)}$.
\end{corollary}
\begin{proof}
Take $C$ as $A$ and $g$ as $h^{-1}$
in the preceding proposition.
Since by \cite[Lemma 3.3]{MMPre2014} or Lemma \ref{lem:klp},
\begin{equation*}
l_2^{l_1(x)}(h(x)) + k_2^{k_1(x)}(h(\sigma_A(x))) -1
=l_2^{k_1(x)}(h(\sigma_A(x))) + k_2^{l_1(x)}(h(x)),
\end{equation*}
we  have
$l_3(x) -1 = k_3(x)$
so that for $f \in C(X_A,\Z)$
the equality
$
\eqref{eq:circ1A}-\eqref{eq:circ2A}
= f(x)
$
holds.
Hence 
$(\Psi_h \circ \Psi_{h^{-1}})(f)(x) = f(x).$
Similarly we have
$
\Psi_{h^{-1}}\circ \Psi_{h} = \id_{C(X_B,\Z)}.
$
\end{proof}

\begin{lemma}
Let $h:(X_A,\sigma_A) \rightarrow (X_B,\sigma_B)$
be a continuous orbit map.
Then we have
\begin{equation} 
\Psi_h(f- f \circ \sigma_B) = f \circ h - f \circ h \circ \sigma_A,
\qquad
f \in C(X_B,\Z). \label{eq:coboundary}
\end{equation}
\end{lemma}
\begin{proof}
Since
\begin{align*}
\sum_{i=0}^{l_1(x)-1} (f- f \circ \sigma_B)(\sigma_B^i(h(x))) 
& = f(h(x)) - f(\sigma_B^{l_1(x)}(h(x))), \\ 
\sum_{j=0}^{k_1(x)-1} (f- f \circ \sigma_B)(\sigma_B^j(h(\sigma_A(x))))
& =f(h(\sigma_A(x))) - f(\sigma_B^{k_1(x)}(h(\sigma_A(x)))),
\end{align*}
the equality
\eqref{eq:coboundary}
follows from
\eqref{eq:orbiteqx}.
\end{proof}
Therefore we have
\begin{proposition}
The homomorphism 
$\Psi_h : C(X_B,\Z) \rightarrow C(X_A,\Z)$
induces a homomorphism
of abelian groups
$\bar{\Psi}_h: H^B \rightarrow H^A$.
\end{proposition}


By a similar argument to the discusssions of \cite[Section 5]{MMPre2014},
one may prove that 
$\Psi_h$ preserves the positive cones of the ordered abelian groups,
that is,
$\bar{\Psi}_h(H^B_+) \subset H^A_+$.
We will briefly state a machineray to prove 
$\bar{\Psi}_h(H^B_+) \subset H^A_+$.
Following 
\cite[Definition 5.5]{MMPre2014},
an eventually periodic point $x \in X_A$
is said to be $(r,s)$-{\it attracting} 
for some $r,s \in \Zp$
if it satisfies the following two conditions:
\begin{enumerate}
\renewcommand{\labelenumi}{(\roman{enumi})}
\item
$\sigma_A^r(x) = \sigma_A^s(x).$ 
\item
For any clopen neighborhood $W \subset X_A$ of $x$,
there exist clopen sets $U, V \subset X_A$ and a homeomorphism
$\varphi: V \rightarrow U$
 such that
{\begin{enumerate}
\item $x \in U \subset V \subset W$.
\item $\varphi(x) = x$.
\item $\sigma_A^r(\varphi(w)) = \sigma_A^s(w)$ for all $w \in V$. 
\item $\lim_{n \to\infty} \varphi^n(w) = x$  for all $w \in V$. 
\end{enumerate}}
\end{enumerate}
Let $x\in X_A$ be an eventually periodic point.
By \cite[Lemma 5.6 and Lemma 5.7]{MMPre2014},
there exists $r, s \in \Zp$ such that 
$x$ is (r,s)-attracting and hence
$\sigma_A^r(x) = \sigma_A^s(x)$, $r >s$.
Since $h:X_A \rightarrow X_B$ is a continuous orbit map,
by using a similar argument to \cite[Lemma 5.8 and Corollary 5.9]{MMPre2014},
one may show that 
$h(x)$ is 
$(l_1^r(x) +k_1^s(x),   k_1^r(x) + l_1^s(x) )$-attracting,
and
hence
$l_1^r(x) +k_1^s(x) > k_1^r(x) + l_1^s(x)$.
Put
$r' = l_1^{q}(\sigma_A^s(x)), s' = k_1^{q}(\sigma_A^s(x))$
where $q = r-s$
and
$z = \sigma_B^{l_1^s(x) +k_1^s(x)}(h(x)) \in X_B$.
By \cite[Lemma 5.3]{MMPre2014}, we then have 
\begin{equation*}
r' - s' =  
(l_1^r(x) - l_1^s(x)) -( k_1^r(x)- k_1^s(x)) >0
\quad
\text{ and }
\quad 
\sigma_B^{r'}(z) 
= 
\sigma_B^{s'}(z).
\end{equation*} 
We set for $f \in C(X_A,\Z)$
\begin{equation*}
\omega_f^{r,s}(x) = 
\sum_{i=0}^{r-1} f(\sigma_A^i(x)) -\sum_{j=0}^{s-1} f(\sigma_A^j(x)).
\label{eq:omegafrs}
\end{equation*}
Then $[f]$ belongs to $H^A_+$ 
if and only if 
$\omega_f^{r,s}(x)>0$
(\cite[Lemma 3.2]{MMKyoto}, \cite[Lemma 5.2]{MMPre2014}).
By \cite[Lemma 5.3]{MMPre2014}, we have
\begin{equation*}
\omega_{\Psi_h(f)}^{r,s}(x)
 = 
\omega_f^{r',s'}(z) \quad \text{ for } \quad f \in C(X_B,\Z). 
\label{eq:oPhfx}
\end{equation*}
Hence  $[f]\in H^B_+$ implies 
$\omega_{\Psi_h(f)}^{r,s}(x)
 = 
\omega_f^{r',s'}(z) >0$
so that
$[\Psi_h(f)] $ belongs to $H^A_+$.

We thus conclude
\begin{theorem}\label{thm:contravariant}
 Let 
$h:(X_A,\sigma_A) \rightarrow (X_B,\sigma_B)$
and
$g:(X_B,\sigma_B) \rightarrow (X_C,\sigma_C)$
be continuous orbit maps.
 Then the homomorphisms
\begin{equation*}
\Psi_h:C(X_B,\Z) \rightarrow C(X_A,\Z),
\qquad 
\Psi_g:C(X_C,\Z) \rightarrow C(X_B,\Z)
\end{equation*}
satisfy the following conditions:
\begin{enumerate}
\renewcommand{\theenumi}{\roman{enumi}}
\renewcommand{\labelenumi}{\textup{(\theenumi)}}
\item
$
\Psi_h \circ \Psi_g = \Psi_{g \circ h}.
$
\item
$\Psi_h(\cobdy(\sigma_B)) \subset \cobdy(\sigma_A)$
and
$\Psi_g(\cobdy(\sigma_C)) \subset \cobdy(\sigma_B).$
\item
They induce homomorphisms 
$\bar{\Psi}_h: (H^B,H^B_+)\rightarrow (H^A,H^A_+)$ 
and  
$\bar{\Psi}_g : (H^C,H^C_+)\rightarrow (H^B,H^B_+)$
of ordered abelian groups such that
$
\bar{\Psi}_h \circ \bar{\Psi}_g = \bar{\Psi}_{g \circ h}.
$
\end{enumerate}
\end{theorem}

\begin{corollary}
The correspondence $\bar{\Psi}$ gives rise to a contravariant functor from the category
$\mathcal{C}_{\COE}$ of the continuous orbit equivalence classes
of one-sided topological Markov shifts
with continuous orbit maps as morphisms
to the category $\mathcal{A}_+$ of ordered abelian groups:
\begin{equation}
[(X_A,\sigma_A)] \in \mathcal{C}_{\COE}
\rightarrow 
(H^A,H^A_+) \in \mathcal{A}_+.
\end{equation}
\end{corollary}

\section{Strongly continuous orbit equivalence}

\begin{definition}\label{defn:storbitmap}
A continuous orbit map  
$h:(X_A,\sigma_A) \rightarrow (X_B,\sigma_B)$ 
is called a {\it strongly continious orbit map\/}
if
there exists a continuous function
$b_1: X_A \rightarrow \Z$
such that 
\begin{equation}
\Psi_h(1_B)(x) = 1 + b_1(x) - b_1(\sigma_A(x)), \qquad x \in X_A.
\label{eq:b1}
\end{equation}
\end{definition}
For a nonnegative integer $N_1$, 
the function $b'_1(x) = b_1(x) +N_1$
still satisfies the above equality.
One may assume that the funtion
$b_1$ in \eqref{eq:b1} is nonnegative.
If a continuous orbit homeomorphism 
is a strongly continuous orbit map,
it is called a {\it strongly continuous orbit homeomorphism\/}.
\begin{definition}\label{defn:storbiteq}
One-sided topological Markov shifts
 $(X_A, \sigma_A)$ and $(X_B,\sigma_B)$
 are said to be 
{\it strongly continuous orbit equivalent\/} 
 if  
 there exists a 
strongly continuous orbit homeomorphism 
$h:(X_A, \sigma_A) \rightarrow (X_B, \sigma_B)$ 
such that its inverse
$h^{-1}:(X_B, \sigma_B) \rightarrow (X_A, \sigma_A)$
is also  
 a strongly continuous orbit homeomorphism. 
In this case, we write 
$(X_A, \sigma_A)\underset{\SCOE}{\sim}(X_B,\sigma_B)$.
\end{definition}

By definition,
 $(X_A, \sigma_A)$ and $(X_B,\sigma_B)$
 are  
strongly continuous orbit equivalent 
 if  and only if 
 there exist a  homeomorphism 
$h:X_A \rightarrow X_B$, 
 continuous functions 
$
k_1,l_1, b_1:X_A\rightarrow \Zp
$
and
$ k_2,l_2, b_2:X_B \rightarrow \Zp
$
such that
\begin{align}
\sigma_B^{k_1(x)} (h(\sigma_A(x))) 
& = \sigma_B^{l_1(x)}(h(x)),
\qquad 
x \in X_A,  \label{eq:scoe1}\\
\sigma_A^{k_2(y)} (h^{-1}(\sigma_B(y))) 
& = \sigma_A^{l_2(y)}(h^{-1}(y)),
\qquad 
y \in X_B, \label{eq:scoe2}\\
\intertext{and}
l_1(x) - k_1(x) 
= & 1 + b_1(x) - b_1(\sigma_A(x)), \qquad x \in X_A,\label{eq:scoe3} \\
l_2(y) - k_2(y) 
= & 1 + b_2(y) - b_2(\sigma_B(y)), \qquad y \in X_B.\label{eq:scoe4} 
\end{align}
Recall that the cocycle functions $c_1, c_2$ are defined by 
$c_1(x) =  l_1(x) - k_1(x)$ for $x \in X_A$ 
and
$c_2(x) =  l_2(y) - k_2(y)$ for $y \in X_B$.
\begin{lemma}
Suppose that $h:X_A \rightarrow X_B$
is a homeomorphism which gives rise to a continuous orbit equivalence
between $(X_A,\sigma_A)$ and $(X_B,\sigma_B)$.
Then the following three conditions are equivalent:
\begin{enumerate}
\renewcommand{\theenumi}{\roman{enumi}}
\renewcommand{\labelenumi}{\textup{(\theenumi)}}
\item
$(X_A, \sigma_A)\underset{\SCOE}{\sim}(X_B,\sigma_B)$.
\item
$[c_1] = [1_A] \in H^A$.
\item
$[c_2] = [1_B] \in H^B$.
\end{enumerate} 
\end{lemma}
\begin{proof}
(ii) $\Rightarrow$ (iii).
Suppose that $[c_1] = [1_A] \in H^A$.
Take a  continuous function 
$b_1 \in C(X_A, \Zp) 
$
such that
$c_1(x)=  1_A(x) + b_1(x) - b_1(\sigma_A(x)),  x \in X_A.
$
Since
$c_1 = \Psi_h(1_B), c_2 = \Psi_{h^{-1}}
(1_A),
$
 we have
$\Psi_{h^{-1}}(c_1) =\Psi_{h^{-1}}(\Psi_h(1_B)) = 1_B$
so that  
\begin{equation*}
c_2 
 = \Psi_{h^{-1}}(c_1 - b_1 + b_1 \circ \sigma_A) 
 = 1_B - \{ \Psi_{h^{-1}}(b_1)  - \Psi_{h^{-1}}(b_1) \circ \sigma_B \}.
\end{equation*}
This implies that
$[c_2] = [1_B] \in H^B$.
(iii) $\Rightarrow$ (ii) is similar.
It is clear that (i) is equivalent to the both conditions (ii) and (iii). 
\end{proof}
Therefore we have
\begin{proposition}
Strongly continuous orbit equivalence is an equivalence relation
in the set of one-sided topological Markov shifts.
\end{proposition}
\begin{proof}
Let
$h:(X_A,\sigma_A) \rightarrow (X_B,\sigma_B)$
and
$g:(X_B,\sigma_B) \rightarrow (X_C,\sigma_C)$
be strongly continuous orbit homeomorphisms.
Put
$c_{AB} = \Psi_h(1_B), c_{BC} =\Psi_g(1_C)$
the cocycle functions for $h$ and $g$ respectively.
By the previous lemma,
we know that the composition 
$g \circ h: X_A \rightarrow X_C$ yields 
a continuous orbit equivalence between
$(X_A,\sigma_A)$ and $(X_C,\sigma_C)$.
Put
$c_{AC} = \Psi_{g\circ h}(1_C)$
so that
$c_{AC} = \Psi_h(\Psi_g(1_C))= \Psi_h(c_{BC})$.
Since
$[c_{AB}] = [1_A], [c_{BC}] = [1_B]$,
one has 
$[c_{AC}] = [\Psi_h(1_B)] = [c_{AB}] =[1_A]$
so that
$(X_A, \sigma_A)\underset{\SCOE}{\sim}(X_C,\sigma_C)$.
\end{proof}

\begin{proposition}\label{prop:SCOEcategory}
The objects $\mathcal{O}_{\SCOE}$ 
of strongly  continuous orbit equivalence classes
of one-sided topological Markov shifts
with the morphisms $\mathcal{M}_{\SCOE}$
 of strongly continuous orbit maps
form a category 
$\mathcal{C}_{\SCOE}=(\mathcal{O}_{\SCOE}, \mathcal{M}_{\SCOE})$.  
\end{proposition}

\section{Two-sided conjugacy} 
Throughout this section,
we assume that 
$h:X_A \rightarrow X_B$
is a homeomorphism which 
gives rise to a strongly continuous orbit equivalence
between $(X_A,\sigma_A)$ and $(X_B,\sigma_B)$.
Let 
$k_1, l_1, b_1 : X_A \rightarrow \Zp$
and
$k_2, l_2, b_2 : X_B \rightarrow \Zp$
be continuous functions 
satisfying 
\eqref{eq:scoe1}, \eqref{eq:scoe3} and \eqref{eq:scoe2}, \eqref{eq:scoe4},
respectively.
\begin{lemma}
Put
$
\varphi_{b_1}(x) = \sigma_B^{b_1(x)}(h(x))
$
for $x \in X_A$.
Then we have
\begin{equation}
\varphi_{b_1}(\sigma_A(x)) = \sigma_B(\varphi_{b_1}(x)), 
\qquad
x \in X_A.
\end{equation}
\end{lemma}
\begin{proof}
It follows that 
\begin{align*}
\varphi_{b_1}(\sigma_A(x)) 
& = \sigma_B^{b_1(\sigma_A(x))}(h(\sigma_A(x))) \\
& = \sigma_B^{1 + b_1(x) -l_1(x)}\sigma_B^{k_1(x)}(h(\sigma_A(x))) \\
& = \sigma_B(\sigma_B^{b_1(x)}(h(x))) \\
& = \sigma_B(\varphi_{b_1}(x)).
\end{align*}
\end{proof}
For $n \in \Zp$,
put $c_1^n(x) = l_1^n(x) - k_1^n(x), x \in X_A$
and
$c_2^n(y) = l_2^n(y) - k_2^n(y), y \in X_B.$
\begin{lemma} Keep the above notations.
\begin{enumerate}
\renewcommand{\theenumi}{\roman{enumi}}
\renewcommand{\labelenumi}{\textup{(\theenumi)}}
\item
$c_1^n(x )  = n + b_1(x) - b_1(\sigma_A^n(x))$ for $x \in X_A$.
\item
$c_2^n(y )  = n + b_2(y) - b_2(\sigma_B^n(y))$ for $y \in X_B$.
\end{enumerate} 
\end{lemma}
\begin{proof}
(i)
As
$l_1(\sigma_A^m(x)) - k_1(\sigma_A^m(x)) 
= 1 + b_1(\sigma_A^m(x)) - b_1(\sigma_A^{m+1}(x)),
$
we have
\begin{equation*}
c_1^n(x) = 
\sum_{m=0}^{n-1}l_1(\sigma_A^m(x))- 
\sum_{m=0}^{n-1}k_1(\sigma_A^m(x))
= n + b_1(x) - b_1(\sigma_A^n(x)).
\end{equation*}
(ii) is similar to (i).
\end{proof}
Since $(X_A,\sigma_A)$ and $(X_B,\sigma_B)$ are 
continuously orbit equivalent,
the following identities hold.
Its proof is seen in \cite[Lemma 3.3]{MMPre2014}.
\begin{lemma}[{\cite[Lemma 3.3]{MMPre2014}}] \label{lem:klp}
For $x \in X_A$, $y \in X_B$ and $p \in \Zp$, 
we have 
\begin{enumerate}
\renewcommand{\theenumi}{\roman{enumi}}
\renewcommand{\labelenumi}{\textup{(\theenumi)}}
\item
$
k_2^{l_1^p(x)}(h(x)) + l_2^{k_1^p(x)}(h(\sigma_A^p(x))) + p
=
k_2^{k_1^p(x)}(h(\sigma_A^p(x))) + l_2^{l_1^p(x)}(h(x)).
$
\item
$
k_1^{l_2^p(y)}(h^{-1}(y)) + l_1^{k_2^p(y)}(h^{-1}(\sigma_B^p(y))) + p
=
k_1^{k_2^p(y)}(h^{-1}(\sigma_B^p(y))) + l_1^{l_2^p(y)}(h^{-1}(y)).
$
\end{enumerate}
\end{lemma} 
Keep the situation.
\begin{lemma} \label{lem:constant}
There exists $N_h \in \N$
such that
$ b_1(x) + b_2(h(x)) =N_h$ for all $x \in X_A$,
and equivalently 
$ b_2(y) + b_1(h^{-1}(y)) =N_h$ for all $y \in X_B$.
\end{lemma}
\begin{proof}
By Lemma \ref{lem:klp},
we have
\begin{equation*}
k_2^{l_1(x)}(h(x)) + l_2^{k_1(x)}(h(\sigma_A(x))) + 1
=
k_2^{k_1(x)}(h(\sigma_A(x))) + l_2^{l_1(x)}(h(x))
\end{equation*}
so that
\begin{equation*}
c_2^{k_1(x)}(h(\sigma_A(x))) + 1
=
c_2^{l_1(x)}(h(x))
\end{equation*}
and hence
\begin{equation*}
k_1(x) + b_2(h(\sigma_A(x))) - b_2(\sigma_B^{k_1(x)}(h(\sigma_A(x) ))) +1
=
l_1(x) + b_2(h(x)) - b_2( \sigma_B^{l_1(x)}(h(x))).
\end{equation*}
As
$
\sigma_B^{k_1(x)}(h(\sigma_A(x))) = \sigma_B^{l_1(x)}(h(x))$,
we have
\begin{equation*}
k_1(x) + b_2(h(\sigma_A(x) ))  +1
=
l_1(x) + b_2(h(x))
\end{equation*}
so that
\begin{equation*}
c_1(x) =  b_2(h(\sigma_A(x) ))  - b_2(h(x)) + 1.
\end{equation*}
Hence we have
\begin{equation*}
b_1(x) - b_1(\sigma_A(x)) =  b_2(h(\sigma_A(x))) - b_2(h(x)).
\end{equation*}
This implies that 
the function
$x \in X_A \rightarrow b_1(x) + b_2(h(x)) \in \N$ 
is $\sigma_A$-invariant,
so that it is constant.
\end{proof}
\begin{theorem}\label{thm:twoconjugacy}
Suppose that
$(X_A, \sigma_A)\underset{\SCOE}{\sim}(X_B,\sigma_B)$.
Then their two-sided topological Markov shifts
$(\bar{X}_A, \bar{\sigma}_A)$ and 
$(\bar{X}_B, \bar{\sigma}_B)$
are topologically conjugate.
 \end{theorem}
\begin{proof}
By the preceding lemma,
the map $\varphi_{b_1}: X_A \rightarrow X_B$
defined by 
$\varphi_{b_1}(x)= \sigma_B^{b_1(x)}(h(x))$
satisfies
\begin{equation}
\varphi_{b_1}(\sigma_A(x)) = \sigma_B(\varphi_{b_1}(x)), 
\qquad
x \in X_A.
\end{equation}
For $\bar{x} = (x_i)_{i \in \Z} \in \bar{X}_A$
and $j \in \Z$,
put 
$
\bar{x}(j) = x_{[j, \infty)}\in X_A.
$
We set
$\bar{y}{[j]} =\varphi_{b_1}(\bar{x}(j))\in X_B$.
It then follows that
\begin{equation*}
\sigma_B(\bar{y}[j])
 =\varphi_{b_1}(\sigma_A(\bar{x}(j))) =\varphi_{b_1}(\bar{x}(j+1))
 =\bar{y}[j+1]. 
\end{equation*}
Hence we may define an element
$\bar{y} = (y_j)_{j \in \Z} \in \bar{X}_B$
such that 
$\bar{y}_{[j,\infty)} = \bar{y}[j]$.
We set
$\bar{h}(\bar{x}) = \bar{y}$
so that
$\bar{h}: \bar{X}_A \rightarrow \bar{X}_B$
is a continuous map.
Since
$(\bar{\sigma}_A(\bar{x}))(j) 
= x_{[j+1,\infty)} = \sigma_A(\bar{x}(j))$,
we have
\begin{equation*}
[\bar{h}(\bar{\sigma}_A(\bar{x}))]_{[j,\infty)}
=\varphi_{b_1}([\bar{\sigma}_A(\bar{x})](j))
 =\varphi_{b_1}(\sigma_A(\bar{x}(j))) 
 =[\bar{h}(\bar{x})]_{[j+1,\infty)}
 =[\bar{\sigma}_B(\bar{h}(\bar{x}))]_{[j,\infty)} 
\end{equation*}
so that
\begin{equation*}
\bar{h}(\bar{\sigma}_A(\bar{x}))
 =
\bar{\sigma}_B(\bar{h}(\bar{x})), \qquad \bar{x} \in \bar{X}_A. 
\end{equation*}
Hence
$\bar{h}: \bar{X}_A \rightarrow \bar{X}_B$
is a sliding block code (see \cite{LM} for the definition of the sliding block code).
One may similarly construct a sliding block code
$\overline{h^{-1}}: \bar{X}_B \rightarrow \bar{X}_A$
from the inverse
$h^{-1}:X_B\rightarrow X_A$ of $h$.
We denote by 
$\psi_{b_2}: X_B \rightarrow X_A$
the continuous map defined by
$\psi_{b_2}(y) = \sigma_A^{b_2(y)}(h^{-1}(y))$,
which satisfies
$
\psi_{b_2}(\sigma_B(y)) = \sigma_A(\psi_{b_2}(y)), 
y \in X_B.
$
Then the map
$\overline{h^{-1}}: \bar{X}_B \rightarrow \bar{X}_A$
satisfies the equality
$(\overline{h^{-1}}(\bar{y}))_{[j,\infty)} = \psi_{b_2}(\bar{y}(j))$
for $j \in \Z$.
It then follows that for $j \in \Z$
\begin{align*}
(\overline{h^{-1}}(\bar{h}(\bar{x})))_{[j,\infty)} 
& = \psi_{b_2}(\varphi_{b_1}(\bar{x}(j))) \\
& = \psi_{b_2}(\sigma_B^{b_1(\bar{x}(j))}(h(\bar{x}(j)))) \\
& = \sigma_A^{b_1(\bar{x}(j))}(\psi_{b_2}(h(\bar{x}(j)))) \\
& = \sigma_A^{b_1(\bar{x}(j))}(\sigma_A^{b_2(h(\bar{x}(j)))}
(h^{-1}(h(\bar{x}(j))))) \\
& = \sigma_A^{b_1(\bar{x}(j)) +b_2(h(\bar{x}(j)))}(\bar{x}(j)) \\
& = (\bar{\sigma}_A^{b_1(\bar{x}(j)) +b_2(h(\bar{x}(j)))}(\bar{x}))_{[j,\infty)}. 
\end{align*}
Take a constant number $N_h$ 
in the preceding lemma
so that  we have
\begin{equation*}
(\overline{h^{-1}}(\bar{h}(\bar{x})))_{[j,\infty)} 
= 
(\bar{\sigma}_A^{N_h}(\bar{x}))_{[j,\infty)} 
\qquad
\text{ for all } j \in \Z 
\end{equation*}
and hence 
\begin{equation*}
\overline{h^{-1}}(\bar{h}(\bar{x})) 
 = \bar{\sigma}_A^{N_h}(\bar{x}) \qquad
\text{ for all } \bar{x} \in \bar{X}_A. 
\end{equation*}
We thereby know that 
 $\bar{h}:\bar{X}_A \rightarrow \bar{X}_B$
is injective.
Similarly we see
\begin{equation*}
\bar{h}(\overline{h^{-1}}(\bar{y})) 
 = \bar{\sigma}_B^{N_h}(\bar{y})\qquad
\text{ for all } \bar{y} \in \bar{X}_B 
\end{equation*}
 so that $\bar{h}:\bar{X}_A \rightarrow \bar{X}_B$
is surjective and gives rise to a topological conjugacy
between
$(\bar{X}_A, \bar{\sigma}_A)$ and 
$(\bar{X}_B, \bar{\sigma}_B)$.
\end{proof}

\begin{corollary}
Suppose that
$(X_A, \sigma_A)\underset{\SCOE}{\sim}(X_B,\sigma_B)$.
Then 
each is a finite factor of the other.
In particular,
$(X_A, \sigma_A)$
and
$(X_B,\sigma_B)$ 
are weakly conjugate.
\end{corollary}
\begin{proof}
Since the
two-sided topological Markov shifts
$(\bar{X}_A, \bar{\sigma}_A)$ and 
$(\bar{X}_B, \bar{\sigma}_B)$
are topologically conjugate,
the assertion that 
each of their one-sided topological Markov shifts
$(X_A, \sigma_A)$
and
$(X_B,\sigma_B)$
is a finite factor of the other 
comes from a general theory of symbolic dynamics
(\cite[Exercise 2]{Kitchens}).
One also knows it from the equalities:
$
\psi_{b_2}\circ \varphi_{b_1} = \sigma_A^{N_h}
$
and
$
\varphi_{b_1} \circ \psi_{b_2} = \sigma_B^{N_h}.
$
\end{proof}
Let us denote by $C(\bar{X}_A)$ 
the commutative $C^*$-algebra 
of complex valued continuous functions on $\bar{X}_A$.
The homeomorphism
$\bar{\sigma}_A$ on $\bar{X}_A$ naturally 
induces an automorhism
$\bar{\sigma}_A^*$ on $C(\bar{X}_A)$ 
by
$\bar{\sigma}_A^*(f) = f \circ \bar{\sigma}_A^{-1}$
for $f \in C(\bar{X}_A)$. 
\begin{corollary} \label{cor:crossedXA}
Suppose that
$(X_A, \sigma_A)\underset{\SCOE}{\sim}(X_B,\sigma_B)$.
Then 
their $C^*$-crossed products are isomorphic:
\begin{equation*}
C(\bar{X}_A) \times_{\bar{\sigma}_A^*} \Z 
\cong
C(\bar{X}_B) \times_{\bar{\sigma}_B^*} \Z. 
\end{equation*} 
\end{corollary}
We note that 
$K_0$-group
$K_0(C(\bar{X}_A) \times_{\bar{\sigma}_A^*} \Z)$
of the $C^*$-algebra
$C(\bar{X}_A) \times_{\bar{\sigma}_A^*} \Z$
is isomorphic to the ordered group
$(\bar{H}^A, \bar{H}^A_+)$ 
(see \cite[Theorem 5.2]{BH}, \cite[Remark 3.10]{Po}).

Let 
$\pi_A:\bar{X}_A \rightarrow X_A$
denote the restriction
defined by
$\pi_A((x_n)_{n \in \Z}) = (x_n)_{n \in \N} \in X_A$  
for $(x_n)_{n \in \Z} \in \bar{X}_A$.
The following proposition is a converse to
Theorem \ref{thm:twoconjugacy}.
\begin{proposition}
Let
$h:X_A \rightarrow X_B$ 
be a homeomorphism
such that there exist 
a topological conjugacy
$\bar{h}: (\bar{X}_A,\bar{\sigma}_A) 
\rightarrow
(\bar{X}_B,\bar{\sigma}_B)
$
as two-sided subshifts
and continuous functions
$f_1:X_A \rightarrow \Zp,$ 
$f_2:X_B \rightarrow \Zp$
such that 
\begin{align*}
\pi_B(\bar{h}(\bar{x})) 
& = \sigma_B^{f_1(x)}(h(x)) \quad 
\text{ for } \quad \bar{x} \in \bar{X}_A, \\
\pi_A(\bar{h}^{-1}(\bar{y})) 
& = \sigma_A^{f_2(y)}(h^{-1}(y)) \quad 
\text{ for } \quad \bar{y} \in \bar{X}_B
\end{align*}
where
$x = \pi_A(\bar{x}), y = \pi_B(\bar{y})$.
 Then 
$h:X_A \rightarrow X_B$ 
gives rise to a strongly continuous orbit homeomorhism.
Hence
we have
$(X_A, \sigma_A)\underset{\SCOE}{\sim}(X_B,\sigma_B)$.
\end{proposition}
\begin{proof}
We note that 
$\bar{h} \circ \bar{\sigma}_A = \bar{\sigma}_B \circ \bar{h}$
and
$\pi_B \circ \bar{\sigma}_B = \sigma_B \circ \pi_B$.
For $\bar{x} \in \bar{X}_A$, we have 
\begin{equation*}
\pi_B(\bar{h}(\bar{\sigma}_A(\bar{x})))
=\sigma_B (\pi_B(\bar{h}(\bar{x})))
=\sigma_B^{f_1(x)+1}(h(x)).
\end{equation*}
As
$\pi_A(\bar{\sigma}_A(\bar{x})) = \sigma_A(x)$,
we also have
\begin{equation*}
\pi_B(\bar{h}(\bar{\sigma}_A(\bar{x})))
=\sigma_B^{f_1(\pi_A(\bar{\sigma}_A(\bar{x})))}
(h(\pi_A(\bar{\sigma}_A(\bar{x}))))
=\sigma_B^{f_1(\sigma_A(x))}(h(\sigma_A(x)))
\end{equation*}
so that
\begin{equation*}
 \sigma_B^{f_1(x)+1}(h(x))
=\sigma_B^{f_1(\sigma_A(x))}(h(\sigma_A(x))), \qquad x \in X_A.
\end{equation*}
This implies that 
$h:X_A \rightarrow X_B$ 
and smilarly
$h^{-1}:X_B \rightarrow X_A$ 
give rise to  strongly continuous orbit maps,
which are the inverses to each other.
\end{proof}

\section{Cocycle conjugacy}
Let us denote by
$S_1, \dots, S_N$ the generating partial isometries
of the Cuntz-Krieger algebra $\OA$ satisfying 
\begin{equation} 
\sum_{j=1}^N S_j S_j^* = 1, \qquad
S_i^* S_i = \sum_{j=1}^N A(i,j) S_jS_j^*, \quad i=1,\dots,N. \label{eq:CK}
\end{equation} 
For $t \in {\mathbb{R}}/\Z = {\mathbb{T}}$,
the correspondence
$S_i \rightarrow e^{2 \pi\sqrt{-1}t}S_i$
gives rise to an automorphism
of $\OA$ which we denote by
$\rho^A_t \in \Aut(\OA)$.
The automorphisms
yield an action of ${\mathbb{T}}$
to $\Aut(\OA)$ 
which we call the gauge action.
Let us denote by $\DA$ the $C^*$-subalgebra of $\OA$
generated by the projections of the form:
$S_{i_1}\cdots S_{i_n}S_{i_n}^* \cdots S_{i_1}^*,
$
which is canonically isomorphic to the commutative 
$C^*$-algebra $C(X_A)$
by identifying the projection
$S_{i_1}\cdots S_{i_n}S_{i_n}^* \cdots S_{i_1}^*
$
with the characteristic function
$\chi_{U_{i_1\cdots i_n}} \in C(X_A)$
of the cylinder set
$U_{i_1\cdots i_n}$
for the word
${i_1\cdots i_n}$.

Throughout the section,
we assume that $(X_A, \sigma_A)\underset{\SCOE}{\sim}(X_B,\sigma_B)$
and
fix a strongly continuous orbit homeomorphism
$h:(X_A, \sigma_A)\rightarrow (X_B,\sigma_B)$
and continuous functions 
$k_1,l_1, b_1: X_A \rightarrow \Zp$
and
$k_2,l_2, b_2: X_B \rightarrow \Zp$
satisfying
\eqref{eq:scoe1}, \eqref{eq:scoe3} and \eqref{eq:scoe2}, \eqref{eq:scoe4},
respectively.
By Lemma \ref{lem:constant}, we have the following lemma.
\begin{lemma} \label{lem:b1b2}
\hspace{6cm}
\begin{enumerate}
\renewcommand{\theenumi}{\roman{enumi}}
\renewcommand{\labelenumi}{\textup{(\theenumi)}}
\item
$
 b_1(x) - b_1(\sigma_A(x)) 
 =
-b_2(h(x)) + b_2(h(\sigma_A(x)))
$
for 
$x \in X_A.$
\item
$ 
b_2(y) - b_2(\sigma_B(y)) 
 =
-b_1(h^{-1}(y)) + b_1(h^{-1}(\sigma_B(y)))
$
for
$
 y \in X_B.
$
\end{enumerate}
\end{lemma}
For $i \in \{1,\dots,N\}$ and 
$x = (x_n)_{n \in \N} \in X_A$,
we write
$ix = (i,x_1,x_2, \dots)$.
\begin{lemma}\label{lem:zb1b2}
For $i \in \{1,2,\dots,N\}$ and $y \in X_B$
satisfying $ih^{-1}(y) \in X_A$,
put $ z = ih^{-1}(y)$. 
Then we have
\begin{equation}
b_1(z) - b_1(\sigma_A(z)) =
 b_2(y) - b_2(h(z)).
\end{equation}
\end{lemma}
\begin{proof}
Since 
$h(\sigma_A(z)) = y$, 
the desired equality follows from Lemma \ref{lem:b1b2} (i).
\end{proof}
Recall that 
$\rho^B$ stands for the gauge action on $\OB$.
Denote by
$U(\OB)$ and 
$U(\DB)$ 
the group of unitaries of $\OB$
and that of  $\DB$
respectively.
A continuous map 
$t \in {\mathbb{R}}/{\mathbb{Z}} = {\mathbb{T}}
\rightarrow 
v_t \in U(\OB)
$
is called a one-cocycle for $\rho^B$
if
it satisfies
$v_{t+s} = v_t \rho^B_t(v_s), t,s \in {\mathbb{T}}$.
Then 
the map 
$t\in \T \rightarrow   \Ad(v_t) \circ \rho^B_t \in \Aut(\OB)$
yields an action called 
a perturbed action of $\rho^B$ by $v$.
Since the function $b_2$ is regarded as a positive element of $\DB$, 
one may define  unitaries
$
u_t^{b_2} = \exp({2\pi \sqrt{-1} t b_2}) \in U(\DB),
t \in \T.
$
As
$\rho_s^B(u_t^{b_2} ) = u_t^{b_2} $ 
for all $s,t \in \T$,
the family
$
\{
u_t^{b_2}
\}_{t \in \T}
$
is a one-cocycle for $\rho^B$.

The following proposition is a generalization of 
\cite[2.17 Proposition]{CK}.
\begin{proposition}\label{prop:cocycleconjugacy}
Suppose  that $(X_A, \sigma_A)\underset{\SCOE}{\sim}(X_B,\sigma_B)$.
Then there exists an isomorphism
$\Phi:\OA \rightarrow \OB$ such that 
\begin{equation*}
\Phi(\DA) = \DB
\quad 
\text{  and }
\quad
\Phi \circ \rho^A_t = \Ad(u_t^{b_2})\circ \rho^B_t \circ \Phi, 
\qquad t \in {\mathbb{T}}.
\end{equation*}
\end{proposition}
\begin{proof}
The proof below follows essentially the proof of 
\cite[Proposition 5.6.]{MaPacific}.
For the sake of completeness, 
we will give the proof in the following way.
Let us denote by 
$\HA$ (resp. $\HB$) the Hilbert space
with its complete orthonormal system
$\{ e_x^A \mid x \in X_A\}$ 
(resp. $\{ e_y^B \mid y \in X_B\}$).
Consider the partial isometries 
$S_i^A, i=1,\dots,N$ on $\HA$
defined by 
\begin{equation}
S_i^A e_x^A =
\begin{cases}
e_{ix}^A & \text{ if } ix \in X_A,\\
0 & \text{ otherwise.}
\end{cases}
\label{eq:siAex}
\end{equation} 
Then the operators
$S_i^A, i=1,\dots,N$ are 
partial isometries 
satisfying the relations
\eqref{eq:CK}. 
For the $M \times M$ matrix $B =[B(i,j)]_{i,j=1}^M$,
we similarly define the 
partial isometries  $S_i^B, i=1,\dots,M$ 
on $\HB$
satisfying the relations
\eqref{eq:CK} for $B$.
Hence one may identify the Cuntz--Krieger algebra 
$\OA$ (resp.  $\OB$) 
with
the $C^*$-algebra 
$C^*(S_1^A,\dots, S_N^A)$
(resp. $C^*(S_1^B,\dots, S_M^B)$)
generated by the 
partial isometries $S_1^A,\dots, S_N^A$
(resp. $S_1^B,\dots, S_M^B$).
For the continuous function 
$k_1:X_A \rightarrow \Zp,
$
let
$K_1 = \Max\{k_1(x) \mid x \in X_A\}$.
By adding 
$K_1 - k_1(x)$ to $k_1(x)$ and $l_1(x)$,
one may assume that
$k_1(x) = K_1$ for all $x \in X_A$.
Define the unitary
$U_h :\HA \rightarrow \HB$
by 
$
U_h e_x^A = e_{h(x)}^B
$
for
$
x \in X_A.
$
We will see that
$\Phi = \Ad (U_h)$
satisfies the desired properties.
We fix $i \in \{1,\dots,N\}$
and
set
$X_B^{(i)} = \{ y \in X_B \mid ih^{-1}(y) \in X_A \}$.
For $ y \in X_B$, one has
\begin{equation*}
U_h S_i^A U_h^* e_y^B = 
\begin{cases}
e_{h(ih^{-1}(y))}^B & \text{ if }   y \in X_B^{(i)},\\
0 & \text{ otherwise. }
\end{cases}
\end{equation*} 
For $y \in X_B^{(i)}$,
put
$z = ih^{-1}(y) \in X_A$.
By the equality $h(\sigma_A(z)) =y$
with \eqref{eq:orbiteqx}, 
one has
$h(z) 
\in \sigma_B^{-l_1(z)}(\sigma_B^{k_1(z)}(y))
=   \sigma_B^{-l_1(z)}(\sigma_B^{K_1}(y))$
and 
\begin{equation}
h(z) = (\mu_1(z),\dots,\mu_{l_1(z)}(z), y_{K_1+1}, y_{K_1+2}, \dots ) 
\label{eq:hz}
\end{equation}
for some
$(\mu_1(z),\dots,\mu_{l_1(z)}(z)) \in B_{l_1(z)}(X_B)$.
Put
$
L_1= \Max\{l_1(z) \mid z = i h^{-1}(y), \, y \in X_B^{(i)} \}.
$
The set 
$$
W^{(i)} =\{(\mu_1(z),\dots,\mu_{l_1(z)}(z)) \in B_{l_1(z)}(X_B) 
\mid z = i h^{-1}(y), \, y \in X_B^{(i)} \}
$$
of words 
is a finite subset of
$W_{L_1}(X_B) = \cup_{j=0}^{L_1}B_j(X_B)$.
For a word $\nu=(\nu_1,\dots, \nu_j) \in W^{(i)}$,
put a clopen set in $X_B^{(i)}$
$$
E_\nu^{(i)} =\{ y \in X_B^{(i)} 
\mid \mu_1(z) = \nu_1,\dots,\mu_{l_1(z)}(z) =\nu_j, \, z = i h^{-1}(y)  \}
$$
and the projections
$$
Q_\nu^{(i)}  = \chi_{E_\nu^{(i)}}
\quad \text{ and }\quad
P^{(i)} = \chi_{X_B^{(i)}}
$$
in $\DB$,
where
$\chi_{E_\nu^{(i)}}
$ 
and 
$
\chi_{X_B^{(i)}}
$
denote the characteristic functions on $X_B$
for the clopen sets 
$
E_\nu^{(i)}
$ 
and 
$
X_B^{(i)}
$
respectively.
Since 
$
X_B^{(i)}
$
is a disjoint union
$
X_B^{(i)} 
= \cup_{\nu \in W^{(i)} }E_\nu^{(i)},
$
we have
$$
P^{(i)} =\sum_{\nu \in W^{(i)}}Q_\nu^{(i)}.
$$
For $y \in X_B^{(i)}$ and
$ \nu \in W^{(i)}$,  
we have
$ y \in E_\nu^{(i)}$ 
if and only if
$Q_\nu^{(i)}e_y^B = e_y^B$.
By \eqref{eq:hz},
we have
\begin{equation}
e_{h(ih^{-1}(y))}^B 
= \sum_{\nu \in W^{(i)}}
 S_{\nu}^B \sum_{\xi\in B_{K_1}(X_B)}{S_\xi^B}^* Q_\nu^{(i)} e_y^B
 \quad
 \text{ for } y \in X_B^{(i)}.
\end{equation}
Hence 
\begin{equation*}
U_h S_i^A U_h^* e_y^B  
= \sum_{\nu \in W^{(i)}}\sum_{\xi\in B_{K_1}(X_B)}
 S_{\nu}^B {S_\xi^B}^* Q_\nu^{(i)} e_y^B
\qquad \text{ for } y \in X_B^{(i)}
\end{equation*}
so that 
\begin{equation*}
U_h S_i^A U_h^*  
= \sum_{\nu \in W^{(i)}} \sum_{\xi\in B_{K_1}(X_B)}
 S_{\nu}^B {S_\xi^B}^* Q_\nu^{(i)}.
\end{equation*}
As $ Q_\nu^{(i)}\in \DB$,
we have
$\Ad(U_h) (S_i^A) \in \OB$  
so that 
$\Ad(U_h) (\OA) \subset \OB.$  
Since $U_h^* = U_{h^{-1}}$, 
we symmetrically have
$\Ad(U_h^*) (\OB) \subset \OA$
so that 
$\Ad(U_h) (\OA) = \OB.$
It is direct to see that 
$\Ad(U_h) (f) = f \circ h^{-1}$ for $ f \in \DA$
from the definition $U_h e_x^A = e_{h(x)}^B, x \in X_A$
so that we have
$\Ad(U_h) (\DA) = \DB.$

We will next show that 
$
\Ad(U_h) \circ \rho^A_t 
= 
\Ad(u_t^{b_2})\circ \rho^B_t \circ \Ad(U_h)
$
for
$
t \in \T.
$
It follows that 
\begin{align*}
(\Ad(u_t^{b_2})\circ \rho^B_t \circ \Ad(U_h)) (S_i^A) e_y^B
& =  u_t^{b_2} \rho_t^B(U_h S_i U_h^*) (u_t^{b_2})^* e_y^B \\
& =  \sum_{\nu \in W^{(i)}} \sum_{\xi\in B_{K_1}(X_B)}
    u_t^{b_2} \rho_t^B(S_{\nu}^B {S_\xi^B}^* Q_\nu^{(i)}) u_{-t}^{b_2} e_y^B. 
\end{align*}
Since
$Q_\nu^{(i)}e_y^B \ne 0$ if and only if
$Q_\nu^{(i)}e_y^B  = e_y^B$ and $\nu_1 = \mu_1(z),\dots, \nu_j = \mu_{l_1(z)}(z)$.
For $y \in E_\nu^{(i)} \cap U_\xi$, we have 
$S_{\nu}^B {S_\xi^B}^* Q_\nu^{(i)}e_y^B = e_{h(z)}^B = e_{h(ih^{-1}(y))}^B$
so that  
\begin{align*}
    & u_t^{b_2} \rho_t^B(S_{\nu}^B {S_\xi^B}^* Q_\nu^{(i)} ) u_{-t}^{b_2} e_y^B \\
=  & \exp({2 \pi \sqrt{-1}(|\nu|-|\xi|)t})
     u_t^{b_2} S_{\nu}^B {S_\xi^B}^* Q_\nu^{(i)} u_{-t}^{b_2} e_y^B \\
=  & \exp({2 \pi \sqrt{-1}(|\nu|-|\xi| - b_2(y))t})
     u_t^{b_2} S_{\nu}^B {S_\xi^B}^* Q_\nu^{(i)}  e_y^B \\
=  &  \exp({2 \pi \sqrt{-1}(l_1(z) - k_1(z) - b_2(y))t})
     u_t^{b_2}   e_{h(ih^{-1}(y))}^B \\
= & \exp({2 \pi \sqrt{-1}(l_1(z) - k_1(z) - b_2(y) + b_2(h(ih^{-1}(y))))t})
     e_{h(ih^{-1}(y))}^B. 
\end{align*}
Lemma \ref{lem:zb1b2} ensures us the equality
$
 b_2(y) - b_2(h(ih^{-1}(y)))=
 b_1(z) - b_1(\sigma_A(z))
$
so that 
\begin{equation*}     
 l_1(z) - k_1(z) - b_2(y) + b_2(h(ih^{-1}(y)))
= c_1(z) -(b_1(z) - b_1(\sigma_A(z))) =1.
\end{equation*}
As $e_{h(ih^{-1}(y))}^B =U_h S_i^A U_h^*e_y^B$,
we have
\begin{equation*}
 u_t^{b_2} \rho_t^B(S_{\nu}^B {S_\xi^B}^* Q_\nu^{(i)} ) u_{-t}^{b_2} e_y^B
=  \exp({2 \pi \sqrt{-1}t})  U_h S_i^A U_h^*e_y^B 
=  \Ad(U_h)( \rho_t^A(S_i^A )) e_y^B.
\end{equation*}
Hence we have
\begin{align*}
(\Ad(u_t^{b_2})\circ \rho^B_t \circ \Ad(U_h) )(S_i^A) e_y^B
& =  \sum_{\nu \in W^{(i)}} \sum_{\xi\in B_{K_1}(X_B)}
 u_t^{b_2} \rho_t^B(S_{\nu}^B {S_\xi^B}^* Q_\nu^{(i)} )u_{-t}^{b_2} e_y^B \\
& =\Ad(U_h)( \rho_t^A(S_i^A )) e_y^B.
\end{align*}
We thus have
\begin{equation*}
(\Ad(u_t^{b_2})\circ \rho^B_t \circ \Ad(U_h)) (S_i^A)  
=\Ad(U_h)( \rho_t^A(S_i^A ))
\end{equation*}
and hence
\begin{equation*}
\Ad(u_t^{b_2})\circ \rho^B_t \circ \Ad(U_h)  =\Ad(U_h) \circ  \rho_t^A
\quad \text{ for } t \in {\mathbb{T}}.
\end{equation*}
By setting 
$\Phi = \Ad(U_h): \OA\rightarrow \OB$,
we have a desired cocycle conjugacy.
\end{proof}

To prove the converse of the above proposition, 
we provide the following lemma.
\begin{lemma}
For a unitary representation $u$ of $\T$ into $\DB$,
there exists 
a continuous function
$f_0 \in C(X_B, \Z)$ such that 
$
u_t= \exp({2\pi \sqrt{-1} t f_0}) 
$
for $t \in \T.$
\end{lemma}
\begin{proof}
For a unitary representation $u$ of $\T$ into $\DB$,
there exists a $*$-homomorphism
$\varphi^u$ 
from the group $C^*$-algebra 
$C^*(\T)$ of $\T$ to $\DB$
in a natural way.
It induces a homomorphism
$\varphi^u_*: K_0(C^*(\T)) \rightarrow K_0(\DB)$
on their K-groups.
Let
$\chi_{\id}$ denote the identity representation
$\chi_{\id}(s) = s, s \in \T$
of $\T$.
As
$K_0(C^*(\T))= \oplus_{\chi \in \hat{\T}} \Z$
and
$K_0(\DB) = C(X_B,\Z)$,
by putting
$f_0 = \varphi^u_*(\chi_{\id}) \in C(X_B,\Z)$,
one has 
$u_t= \exp({2\pi \sqrt{-1} t f_0}) 
$ for all
$t \in \T.
$
\end{proof}
We thus have the converse of the above proposition
in the following way.
\begin{proposition}
If there exist
a unitary representation $u$ of $\T$ into $\DB$
and an isomorphism 
$\varPhi:\OA \rightarrow \OB$ such that 
$\varPhi(\DA) = \DB$
and
$
\varPhi \circ \rho^A_t = \Ad(u_t)\circ \rho^B_t \circ \varPhi 
$
for
$ t \in {\mathbb{T}}$,
then 
$(X_A, \sigma_A)\underset{\SCOE}{\sim}(X_B,\sigma_B)$.
\end{proposition}
\begin{proof}
For the unitary representation
$u_t \in U(\DB)$,
one may take
a continuous function
$f_0 \in C(X_B, \Z)$ such that 
$
u_t= \exp({2\pi \sqrt{-1} t f_0}),
t \in \T.
$
Represent the algebras $\OA$ on 
$\HA$ and $\OB$ on $\HB$ by \eqref{eq:siAex}.
Since
the isomorphism 
$\varPhi:\OA \rightarrow \OB$ satisfies  
$\varPhi(\DA) = \DB$,
the one-sided topological Markov shifts
$(X_A, \sigma_A)$
and
$(X_B,\sigma_B)$
are continuously orbit equivalent
by 
\cite{MaPacific},
so that there exists a continuous orbit
homeomorphism
$h:X_A \rightarrow X_B$
such that
$\varPhi = \Ad(U_h)$,
where 
$U_h :\HA \rightarrow \HB$ 
is the unitary defined 
by 
$U_he_x^A = e_{h(x)}^B$ for $x \in X_A$.
Let 
$k_1:X_A\rightarrow \Zp$
and
$l_1:X_B\rightarrow \Zp$
be continuous functions satisfying
\eqref{eq:orbiteq1x}. 
For $ i=1,\dots, N$ and $y \in X_B^{(i)}$
put
$z = ih^{-1}(y) \in X_A$.
As in the proof of the preceding proposition, 
one sees that
\begin{align*}
  & (\Ad(u_t)\circ \rho^B_t \circ \Ad(U_h)) (S_i^A) e_y^B \\
= & \exp({2 \pi \sqrt{-1}(l_1(z) - k_1(z) - f_0(y) + f_0(h(ih^{-1}(y))))t})
     e_{h(ih^{-1}(y))}^B 
\end{align*}
and
\begin{equation*}     
\Ad(U_h)( \rho_t^A(S_i^A )) e_y^B
=\exp({2 \pi \sqrt{-1}t}) e_{h(ih^{-1}(y))}^B.
\end{equation*}
Since
$
\Ad(U_h) \circ \rho^A_t 
= 
\Ad(u_t)\circ \rho^B_t \circ \Ad(U_h)
$
for
$ 
t \in {\mathbb{T}},
$
it follows that
\begin{equation*} 
l_1(z) - k_1(z) - f_0(y) + f_0(h(ih^{-1}(y)))-1 =0.
\end{equation*} 
By putting
$b_1(z) = f_0(h(z))$, we have
$ 
l_1(z) - k_1(z) = 1 + b_1(z) - b_1(\sigma_A(z))
$ 
so that 
$(X_A, \sigma_A)\underset{\SCOE}{\sim}(X_B,\sigma_B)$.
\end{proof}
We note the following lemma.
\begin{lemma}
Let $v_t \in U(\OB), t \in \T$
be a one-cocycle for the gauge action $\rho^B$ on $\OB$.
If there exists an isomorphism
$\Psi:\OA \rightarrow \OB$
such that 
$\Psi(\DA) = \DB$
and
$
\Psi \circ \rho^A_t = \Ad(v_t)\circ \rho^B_t \circ \Psi 
$
for
$ t \in {\mathbb{T}}$,
then 
$v_t$ belongs to $\DB$ and hence
$v_{t+s} = v_t v_s, t,s \in \T$.
\end{lemma}
\begin{proof}
For $f \in \DA$,
we have
$
\Psi(\rho^A_t(f)) = v_t(\rho^B_t(\Psi(f)))v_t^*. 
$
As
$
\rho^A_t(f) =f
$
and
$
\rho^B_t(\Psi(f))
= \Psi(f)$,
we see that
$\Psi(f) v_t = v_t \Psi(f)$.
Since 
$\Psi(\DA) = \DB$ 
and
$\DB$ is a maximal commutative $C^*$-subalgebra 
of $\OB$,
the unitarye
$v_t$  belongs to $\DB$.
\end{proof}
Consequently  we have
the following theorem.
\begin{theorem}\label{thm:TFAE}
The following two assertions are equivalent.
\begin{enumerate}
\renewcommand{\theenumi}{\roman{enumi}}
\renewcommand{\labelenumi}{\textup{(\theenumi)}}
\item
One-sided topological Markov shifts
$(X_A, \sigma_A)$ and $(X_B,\sigma_B)$
are strongly continuous orbit equivalent.
\item
There exist
a unitary one-cocycle $v_t \in \OB, t \in \T$
for the gauge action $\rho^B$ on $\OB$
and
 an isomorphism
$\Phi:\OA \rightarrow \OB$ such that 
\begin{equation*}
\Phi(\DA) = \DB
\quad 
\text{  and }
\quad
\Phi \circ \rho^A_t = \Ad(v_t)\circ \rho^B_t \circ \Phi, 
\qquad t \in {\mathbb{T}}.
\end{equation*}
\end{enumerate}
\end{theorem}
As it is well-known that
a cocycle conjugate covariant system of a locally compact abelian group
yields a conjugate dual covariant system,
we have the following corollary.
\begin{corollary}\label{cor:crossedgauge}
Assume that $(X_A, \sigma_A)\underset{\SCOE}{\sim}(X_B,\sigma_B)$.
Then the dual actions of their $C^*$-crossed products 
are isomorphic:
\begin{equation*}
(\OA{\times}_{\rho^A}{\mathbb{T}}, \hat{\rho}^A, \Z)
 \cong
(\OB{\times}_{\rho^B}{\mathbb{T}}, \hat{\rho}^B, \Z).
\end{equation*}
\end{corollary}

\section{Examples}
{\bf 1.} 
Let $A$ and $B$ be the following matrices:
\begin{equation}
A=
\begin{bmatrix}
1 & 1 \\
1 & 1
\end{bmatrix},
\qquad
B=
\begin{bmatrix}
1 & 1 \\
1 & 0 
\end{bmatrix}.
\end{equation}
They are both irreducible and satisfy condition (I).
The one-sided topological Markov shifts 
$(X_A,\sigma_A)$ and $(X_B,\sigma_B)$
are continuously orbit equivalent as in \cite{MaPacific}.
This continuous orbit equivalence  also
comes from the fact that
their Cuntz-Krieger algebras $\OA$ and $\OB$ are isomorphic and
$\det(\id -A) = \det(\id -B)$ by \cite{MMKyoto}.
Since their Perron eigenvalues of $A$ and of $B$ 
are different, the topological entropy of
the two-sided topological Markov shifts
$(\bar{X}_A, \bar{\sigma}_A)$
and
$(\bar{X}_B, \bar{\sigma}_B)$
are different so that they are not topologically conjugate as
two-sided subshifts.

\medskip

{\bf 2.}
If one-sided topological Markov shifts 
$(X_A, \sigma_A)$ and 
$(X_B, \sigma_B)$
are topologically conjugate,
one may take a homeomorphism
$h: X_A \rightarrow X_B$ 
such that $k_1(x) =0, l_1(x)=1$ for all $x \in X_A$,
so that
$c_1(x) =1$ for all $x \in X_A$
and hence
$(X_A, \sigma_A)$ and 
$(X_B, \sigma_B)$
are strongly continuous orbit equivalent. 
We will present an example of 
one-sided topological Markov shifts 
$(X_A, \sigma_A)$ and 
$(X_B, \sigma_B)$
such that 
they are not topologically conjugate,
but they are strongly continuous orbit equivalent. 
Let $A$ and $B$ be the following matrices:
\begin{equation}
A=
\begin{bmatrix}
1 & 1 \\
1 & 1
\end{bmatrix},
\qquad
B=
\begin{bmatrix}
1 & 1& 0 \\
1 & 0& 1 \\
1 & 0& 1 
\end{bmatrix}.
\end{equation}
They are both irreducible and satisfy condition (I).
Since the total column amalgamation of $B$ is it-self,
their one-sided topological Markov shifts
$(X_A, \sigma_A)$ and 
$(X_B, \sigma_B)$
are not topologically conjugate (\cite{Kitchens}, \cite{Wi}). 
We will show the following theorem.
\begin{theorem} \label{thm:ExSCOE}
The one-sided topological Markov shifts 
$(X_A, \sigma_A)$ and 
$(X_B, \sigma_B)$
are strongly continuous orbit equivalent.
\end{theorem}
We will prove Theorem \ref{thm:ExSCOE} as follows.
Let us denote by 
$\Sigma_A = \{\alpha,\beta\}$
the symbols of the shift space $X_A$,
and
similarly
$\Sigma_B = \{ 1, 2, 3 \}$
those of $(X_B, \sigma_B)$,
 respectively.
We note that 
\begin{align*}
B_2(X_A) & = \{ (\alpha,\alpha), (\alpha,\beta),
              (\beta,\alpha), (\beta,\beta) \}, \\
B_2(X_B) & = \{ (1,1), (1,2),( 2,1), (2,3),(3,1), (3,3)\}.
\end{align*}
 Define the block maps
$\Phi$ and $ \varphi$
by
\begin{gather*}
\Phi(\alpha, \alpha) = (1,1), \quad 
\Phi(\beta, \beta, \alpha) = (2,1), \quad 
\Phi(\beta, \alpha, \alpha) = (3,1),  \\
\Phi(\alpha, \beta) = (1,2), \quad
\Phi(\beta, \beta, \beta) = (2,3), \quad 
\Phi(\beta, \alpha, \beta) = (3,3)
\end{gather*}
and
\begin{equation*}
\phi(\alpha, \beta) = 2, \quad
\phi(\beta, \beta) = 3, \quad
\phi(\alpha, \alpha) = 1, \quad
\phi(\beta, \alpha) = 1.
\end{equation*}
It is direct to see that 
the $2$-block map
$\phi : B_2(X_A) \rightarrow B_1(X_B)$
gives rise to a sliding block code
from $X_A$ to $X_B$.
For  $k,l\in \Zp$,
we denote by $\phi_\infty^{[-k,l]}$ 
the sliding block code with memory $k$ and anticipation $l$
(see \cite{LM}).
Define $h:X_A \rightarrow X_B$ by setting
for $x = (x_n)_{n \in \N} \in X_A$
\begin{align*}
  h(x_1,x_2,x_3,\dots ) 
=&
{\begin{cases}
(\Phi(x_1,x_2), \phi_\infty^{[0,1]}(x_2,x_3,\dots ) ) & \text{ if } 
x_1 = \alpha, \\
(\Phi(x_1,x_2,x_3), \phi_\infty^{[-1,0]}(x_3,x_4,\dots ) ) & \text{ if } 
x_1 = \beta  
\end{cases}} \\
=&
{\begin{cases}
(1,1, \phi_\infty^{[0,1]}(\sigma_A(x) )) & \text{ if } (x_1, x_2)=(\alpha, \alpha),\\
(1,2, \phi_\infty^{[0,1]}(\sigma_A(x) )) & \text{ if } (x_1, x_2)=(\alpha, \beta),\\
(2,1, \phi_\infty^{[-1,0]}(\sigma_A^2(x) )) & \text{ if } 
                  (x_1, x_2, x_3)=(\beta, \beta, \alpha), \\
(2,3, \phi_\infty^{[-1,0]}(\sigma_A^2(x) )) & \text{ if } 
                  (x_1, x_2, x_3)=(\beta, \beta, \beta), \\ 
(3,1, \phi_\infty^{[-1,0]}(\sigma_A^2(x) )) & \text{ if } 
                  (x_1, x_2, x_3)=(\beta, \alpha, \alpha),\\
(3,3, \phi_\infty^{[-1,0]}(\sigma_A^2(x) )) & \text{ if } 
                  (x_1, x_2, x_3)=(\beta, \alpha, \beta). 
\end{cases}} 
\end{align*}
We note that $h(x)$ belongs to $X_B$ for all $x\in X_A$ 
because of the following equalities, where
$h(x)_{[1,3]}$ denotes the first three symbols of $h(x)$.
\begin{equation*}
h(x)_{[1,3]}
=
\begin{cases}
(1,1,1) &\text{ if } (x_1, x_2, x_3) = (\alpha, \alpha, \alpha),\\
(1,1,2) &\text{ if } (x_1, x_2, x_3) = (\alpha, \alpha, \beta), \\
(1,2,1) &\text{ if } (x_1, x_2, x_3) = (\alpha, \beta, \alpha),\\
(1,2,3) &\text{ if } (x_1, x_2, x_3) = (\alpha, \beta, \beta), \\ 
(2,1,1) &\text{ if } (x_1, x_2, x_3, x_4) = (\beta, \beta, \alpha, \alpha),\\
(2,1,2) &\text{ if } (x_1, x_2, x_3, x_4) = (\beta, \beta, \alpha, \beta),\\
(2,3,1) &\text{ if } (x_1, x_2, x_3, x_4) = (\beta, \beta, \beta, \alpha),\\
(2,3,3) &\text{ if } (x_1, x_2, x_3, x_4) = (\beta, \beta, \beta, \beta), \\
(3,1,1) &\text{ if } (x_1, x_2, x_3, x_4) = (\beta, \alpha, \alpha, \alpha),\\
(3,1,2) &\text{ if } (x_1, x_2, x_3, x_4) = (\beta, \alpha, \alpha, \beta), \\
(3,3,1) &\text{ if } (x_1, x_2, x_3, x_4) = (\beta, \alpha, \beta, \alpha),\\
(3,3,3) &\text{ if } (x_1, x_2, x_3, x_4) = (\beta, \alpha, \beta, \beta).  
\end{cases}
\end{equation*}
We set
\begin{equation*}
l_1(x)
=
\begin{cases}
1 &\text{ if } (x_1,x_2) = (\alpha, \alpha), \\
4 &\text{ if } (x_1,x_2) = (\alpha, \beta), \\
2 &\text{ if } (x_1,x_2) = (\beta, \alpha),\\
3 &\text{ if } (x_1,x_2) = (\beta, \beta), 
\end{cases}
\qquad
k_1(x)
=
\begin{cases}
0 &\text{ if } (x_1,x_2) = (\alpha, \alpha), \\
2 &\text{ if } (x_1,x_2) = (\alpha, \beta), \\
2 &\text{ if } (x_1,x_2) = (\beta, \alpha),\\
2 &\text{ if } (x_1,x_2) = (\beta, \beta) 
\end{cases}
\end{equation*}
so that we have
\begin{equation*}
\sigma_B^{k_1(x)} (h(\sigma_A(x))) = \sigma_B^{l_1(x)}(h(x))
\quad \text{ for}
\quad x \in X_A
\end{equation*}
and
\begin{equation*}
c_1(x)
=
\begin{cases}
1 &\text{ if } (x_1,x_2) = (\alpha, \alpha), \\
2 &\text{ if } (x_1,x_2) = (\alpha, \beta), \\
0 &\text{ if } (x_1,x_2) = (\beta, \alpha),\\
1 &\text{ if } (x_1,x_2) = (\beta, \beta). 
\end{cases}
\end{equation*}
Define 
a continuous function $b_1: X_A\rightarrow \N$
by
\begin{equation*}
b_1(x)
=
\begin{cases}
2 &\text{ if } x_1 = \alpha,\\
1 &\text{ if } x_1 = \beta. 
\end{cases}
\end{equation*}
Since
\begin{equation*}
b_1(x) - b_1(\sigma_A(x))
=
\begin{cases}
0= c_1(x) -1 &\text{ if } (x_1,x_2) = (\alpha, \alpha), \\
1= c_1(x) -1 &\text{ if } (x_1,x_2) = (\alpha, \beta), \\
-1= c_1(x) -1 &\text{ if } (x_1,x_2) = (\beta, \alpha),\\
 0= c_1(x) -1 &\text{ if } (x_1,x_2) = (\beta, \beta), 
\end{cases}
\end{equation*}
we have
\begin{equation*}
c_1(x) = 1 +b_1(x) - b_1(\sigma_A(x)), \qquad x \in X_A.
\end{equation*}
This implies the following lemma.
\begin{lemma}
 $h:X_A \rightarrow X_B$ is a strongly continuous orbit map.
\end{lemma}

We will next constuct the inverse of $h$.
 Define the block maps
$\Psi$ and $\psi$
by
\begin{gather*}
\Psi(1,1) =(\alpha, \alpha), \quad
\Psi(2,1) =(\beta, \beta, \alpha),  \quad
\Psi(3,1) =(\beta, \alpha, \alpha),  \\
\Psi(1,2) =(\alpha, \beta),  \quad 
\Psi(2,3) =(\beta, \beta, \beta),  \quad
\Psi(3,3) =(\beta, \alpha, \beta)
\end{gather*}
and
\begin{equation*}
\psi(1) = \alpha, \quad
\psi(2) = \beta, \quad
\psi(3) = \beta.
\end{equation*}
It is direct to see that 
the $1$-block map
$\psi : B_1(X_A) \rightarrow B_1(X_B)$
gives rise to a sliding block code
from $X_B$ to $X_A$.
Define $g:X_B \rightarrow X_A$ 
by setting for $y = (y_n)_{n \in \N} \in X_B$
\begin{align*}
  g(y_1,y_2,y_3,y_4, \dots ) 
=&
{\begin{cases}
(\Psi(y_1,y_2), \psi_\infty^{[0,0]}(y_3,y_4,y_5,\dots ) ) & \text{ if } 
y_1 = 1, \\
(\Psi(y_1,y_2), \psi_\infty^{[-1,-1]}(y_3,y_4,y_5,\dots ) ) & \text{ if } 
y_1 = 2,3  
\end{cases}} \\
=&
{\begin{cases}
(\alpha, \alpha, \psi_\infty^{[0,0]}(\sigma_B^2(y) ) & \text{ if } (y_1, y_2)=(1,1),\\
(\alpha, \beta, \psi_\infty^{[0,0]}(\sigma_B^2(y) ) & \text{ if } (y_1, y_2)=(1,2),\\
(\beta, \beta, \alpha, \psi_\infty^{[-1,-1]}(\sigma_B^2(y) ) & \text{ if } (y_1, y_2)=(2,1),\\
(\beta, \beta, \beta, \psi_\infty^{[-1,-1]}(\sigma_B^2(y) ) & \text{ if } (y_1, y_2)=(2,3),\\
(\beta, \alpha, \alpha, \psi_\infty^{[-1,-1]}(\sigma_B^2(y) ) & \text{ if } (y_1, y_2)=(3,1),\\
(\beta, \alpha, \beta, \psi_\infty^{[-1,-1]}(\sigma_B^2(y) ) & \text{ if } (y_1, y_2)=(3,3).
\end{cases}} 
\end{align*}
We set
\begin{align*}
l_2(y)
& =
{\begin{cases}
3 &\text{ if } (y_1, y_2)=(1,1), (1,2), \\
4 &\text{ if } (y_1, y_2)=(2,1), (2,3), (3,1), (3,3),\\
\end{cases}} \\
k_2(y)
& =
{\begin{cases}
2 &\text{ if } (y_1, y_2)=(1,1), (2,1), (3,1), \\
3 &\text{ if } (y_1, y_2)=(1,2), (2,3), (3,3), \\
\end{cases}}
\end{align*}
so that we have
\begin{equation*}
\sigma_A^{k_2(y)} (g(\sigma_B(y))) = \sigma_A^{l_2(y)}(g(y))
\quad \text{ for}
\quad y \in X_B
\end{equation*}
and
\begin{equation*}
c_2(y)
=
\begin{cases}
1 &\text{ if } (y_1, y_2)=(1,1), (2,3), (3,3), \\
0 &\text{ if } (y_1, y_2)=(1,2), \\
2 &\text{ if } (y_1, y_2)=(2,1), (3,1).
\end{cases}
\end{equation*}
Define 
a continuous function $b_2 : X_B\rightarrow \N$
by
\begin{equation*}
b_2(y)
=
\begin{cases}
1 &\text{ if } y_1 =1,\\
2 &\text{ if } y_1 =2, 3.
\end{cases}
\end{equation*}
Since
\begin{equation*}
b_2(y) - b_2(\sigma_B(y))
=
\begin{cases}
 0= c_2(y) -1 &\text{ if } (y_1, y_2)=(1,1), \\
-1= c_2(y) -1 &\text{ if } (y_1, y_2)=(1,2), \\
 1= c_2(y) -1 &\text{ if } (y_1, y_2)=(2,1), (3,1),\\
 0= c_2(y) -1 &\text{ if } (y_1, y_2)=(2,3), (3,3),
\end{cases}
\end{equation*}
we have
\begin{equation*}
c_2(y) = 1 +b_2(y) - b_2(\sigma_B(y)), \qquad y \in X_B.
\end{equation*}
This implies the following lemma.
\begin{lemma}
$g:X_B \rightarrow X_A$ is a strongly continuous orbit map.
\end{lemma}

We will next show that 
$g,h$ are inverses to each other.

For 
$x_1 = \alpha$, we see
\begin{equation*}
\Psi(\Phi(\alpha,x_2))
 = 
{\begin{cases}
\Psi(1,1) =(\alpha, \alpha) & \text{ if } x_2 = \alpha, \\
\Psi(1,2) =(\alpha, \beta) & \text{ if } x_2 = \beta
\end{cases}} 
\end{equation*}
so that 
$\Psi(\Phi(x_1,x_2)) = (x_1,x_2)$.

For
$x_1 = \beta$, we see
\begin{equation*}
 \Psi(\Phi(\beta,x_2, x_3)) 
 = 
{\begin{cases}
\Psi(2,1)=(\beta, \beta, \alpha) & \text{ if }  (x_2,x_3) =(\beta,  \alpha), \\
\Psi(2,3)= (\beta, \beta, \beta) & \text{ if }  (x_2,x_3) =(\beta, \beta), \\
\Psi(3,1)=(\beta, \alpha, \alpha) & \text{ if } (x_2,x_3) =(\alpha, \alpha), \\
\Psi(3,3)=(\beta, \alpha, \beta)  & \text{ if } (x_2,x_3) =(\alpha, \beta)
\end{cases}} 
 \end{equation*}
so that 
$\Psi(\Phi(x_1,x_2,x_3)) = (x_1,x_2,x_3)$.
It is easy to see that
the equalities
\begin{align*}
\psi(\phi(\alpha,x_1,x_2,\dots ))& = (x_1,x_2,\dots ), \\
\psi(\phi(\beta,x_1,x_2,\dots ))& = (x_1,x_2,\dots ) 
\end{align*}
hold so that
$
\psi\circ \phi = \sigma_A
$
on
$
X_A.
$
\begin{lemma}
$g(h(x)) = x$  for  $x \in X_A$.
\end{lemma}
\begin{proof}
It follows that
\begin{align*}
g(h(x_1,x_2,x_3, \dots )) 
& = 
{\begin{cases}
g(\Phi(x_1,x_2), \phi(x_2,x_3, \dots ))  & \text{ if } x_1 = \alpha, \\
g(\Phi(x_1,x_2, x_3), \phi(x_3,x_4, \dots )) & \text{ if } x_1 = \beta 
\end{cases}} \\  
& = 
{\begin{cases}
(\Psi(\Phi(x_1,x_2)), \psi(\phi(x_2,x_3, \dots )))  
& \text{ if } x_1 = \alpha, \\
(\Psi(\Phi(x_1,x_2, x_3)), \psi(\phi(x_3,x_4, \dots ))) 
& \text{ if } x_1 = \beta 
\end{cases}} \\  
& = 
{\begin{cases}
(x_1,x_2, \sigma_A(x_2,x_3, \dots ))  
& \text{ if } x_1 = \alpha, \\
(x_1,x_2, x_3, \sigma_A(x_3,x_4, \dots )) 
& \text{ if } x_1 = \beta 
\end{cases}} \\  
& =(x_1, x_2, x_3, x_4, \dots ).
\end{align*}
\end{proof} 
We will finally prove that 
$h(g(y)) = y$ for all $y=(y_n)_{n \in \N} \in X_B$.
It is direct to see that
\begin{equation*}
\Phi(\Psi(y_1,y_2)) = (y_1,y_2) \qquad \text{ for }
\quad (y_1,y_2) \in B_2(X_B).
\end{equation*}
We have 
\begin{align*}
\phi(\alpha, \psi(y_3,y_4, \dots )) 
&= (y_3, y_4,\dots ) \quad \text{ if } 
y_2 = 1, \\
\phi(\beta, \psi(y_3,y_4, \dots )) 
&= (y_3, y_4,\dots ) \quad \text{ if } 
y_2 = 2, 3.
\end{align*}
We set
$g(y) = (x_n)_{n \in \N} \in X_A$ 
As
\begin{align*}
(x_1,x_2) 
&=
{\begin{cases}
(\alpha, \alpha) & \text{ if } (y_1, y_2) =(1, 1),\\
(\alpha, \beta) & \text{ if }  (y_1, y_2) =(1, 2),
\end{cases}} \\
(x_1,x_2,x_3)
&=
{\begin{cases}
(\beta, \beta, \alpha) & \text{ if }  (y_1, y_2) =(2, 1),\\
(\beta, \beta, \beta) & \text{ if }   (y_1, y_2) =(2, 3),\\
(\beta, \alpha, \alpha) & \text{ if } (y_1, y_2) =(3, 1),\\
(\beta, \alpha, \beta) & \text{ if }  (y_1, y_2) =(3, 3),
\end{cases}} 
\end{align*}
we have 
\begin{align*}
h(g(y))
&
= h(\Psi(y_1,y_2), \psi(y_3,y_4,\dots )) \\
&
=
{\begin{cases}
(\Phi(\Psi(y_1,y_2)), \phi(x_2, \psi(y_3,y_4,\dots )))
& \text{ if } (y_1,y_2) = (1,1), (1,2),\\
(\Phi(\Psi(y_1,y_2)), \phi(x_3, \psi(y_3,y_4,\dots )))
& \text{ if } (y_1,y_2) = (2,1),(2,3), (3,1),(3,3)
\end{cases}} \\
&
=
{\begin{cases}
(\Phi(\Psi(y_1,y_2)), \phi(\alpha, \psi(y_3,y_4,\dots )))
& \text{ if } y_2 = 1, \\
(\Phi(\Psi(y_1,y_2)), \phi(\beta, \psi(y_3,y_4,\dots )))
& \text{ if } y_2 = 2,3
\end{cases}} \\
&
= (y_1, y_2, y_3,y_4,\dots ).
\end{align*}
Hence 
$h(g(y))=y$ for all $y \in X_B$
so that
$g = h^{-1}$,
and hence  
$(X_A, \sigma_A)\underset{\SCOE}{\sim}(X_B,\sigma_B)$.

\medskip

{\bf 3.}
We will finally present 
an example of two irreducible matrices with entries in 
$\{0,1\}$ whose two-sided topological Markov shifts 
are topologically conjugate,
but whose one-sided topological Markov shifts 
are not strongly continuous orbit equivalent.
 Let $A$ and $B$ be the following matrices:
\begin{equation}
A=
\begin{bmatrix}
1 & 1 & 1\\
1 & 1 & 1\\
1 & 0 & 0
\end{bmatrix},
\qquad
B= A^{t}=
\begin{bmatrix}
1 & 1 & 1\\
1 & 1 & 0\\
1 & 1 & 0
\end{bmatrix}.
\end{equation}
They are both irreducible and satisfy condition (I).
Since the row amalgamation of $A$ and the column amalgamation of $B$
are both
$
\begin{bmatrix}
2 & 1 \\
1 & 0
\end{bmatrix},
$
the two-sided topological Markov shifts 
$(\bar{X}_A,\bar{\sigma}_A)$ and $(\bar{X}_B,\bar{\sigma}_B)$
are topologically conjugate (cf. \cite{Kitchens}).
We however know that
${\mathcal{O}}_A\cong {\mathcal{O}}_3$
and
${\mathcal{O}}_B\cong {\mathcal{O}}_3\otimes M_2$
(\cite{EFW}).
Hence their Cuntz--Krieger algebras are not isomorphic
so that
the one-sided topological Markov shifts 
$(X_A,\sigma_A)$ and $(X_B,\sigma_B)$
are not continuously orbit equivalent.


\medskip

{\it Acknowledgments:}
The author would like to thank 
Hiroki Matui for his discussions on this subject.
This work was supported by JSPS KAKENHI Grant Numbers 23540237.



\begin{thebibliography}{99}














\bibitem{BH}
{\sc M. Boyle and D. Handelman},
{\it Orbit equivalence, flow equivalence and ordered cohomology},
Israel J.\  Math.\
{\bf 95}(1996), pp.\ 169--210.

























\bibitem{CK}{\sc J. ~Cuntz and W. ~Krieger},
{\it A class of $C^*$-algebras and topological Markov chains},
 Invent.\ Math.\
 {\bf 56}(1980), pp.\ 251--268.






\bibitem{EFW} {\sc M. Enomoto, M. Fujii and Y. Watatani},
{\it $K_0$-groups and classifications of Cuntz--Krieger algebras}, 
Math.\ Japon. {\bf 26}(1981), pp.\ 443--460.




\bibitem{GPS} {\sc T. Giordano, I. F. Putnam and C. F. Skau},
{\it Topological orbit equivalence and $C^*$-crossed products},
J.\ Reine Angew.\ Math.\ {\bf 469}(1995), pp.\ 51--111.


\bibitem{GPS2} {\sc T. Giordano, I. F. Putnam and C. F. Skau},
{\it Full groups of Cantor minimal systems},
Israel  J.\ Math.\ {\bf 111}(1999), pp.\ 285--320.

\bibitem{GMPS}
{\sc T. Giordano, H. Matui, I. F. Putnam and C. F. Skau},
{\it Orbit equivalemce for Cantor minimal ${\Bbb Z}^2$-systems},
J. \ Amer. \ Math.\  Soc.\  {\bf 21}(2008), pp. \ 863--892.














\bibitem{Kitchens}{\sc B.~P. ~Kitchens},
{\it Symbolic dynamics}, 
Springer-Verlag, Berlin, Heidelberg and New York (1998).






\bibitem{LM}{\sc D. ~Lind and B. ~Marcus},
{\it An introduction to symbolic dynamics and coding},
 Cambridge University Press, Cambridge
(1995).







\bibitem{MaPacific}
{\sc K. Matsumoto},
{\it Orbit equivalence of topological Markov shifts and Cuntz--Krieger algebras},
Pacific J.\ Math.\ 
{\bf 246}(2010), 199--225.













\bibitem{MaPAMS}
{\sc K. Matsumoto},
{\it  Classification of Cuntz--Krieger algebras by orbit equivalence of topological Markov shifts},
Proc. Amer. Math. Soc.
{\bf 141}(2013), pp.\ 2329--2342.



\bibitem{MaPre2012}
{\sc K. Matsumoto},
{\it Full groups of one-sided topological Markov shifts},
preprint, arXiv:1205.1320, to appear in Israel J. Math..

\bibitem{MMKyoto}
{\sc K. Matsumoto and H. Matui},
{\it Continuous orbit equivalence of topological Markov shifts 
and Cuntz--Krieger algebras},
preprint, arXiv:1307.1299, to appear in Kyoto J. Math..


\bibitem{MMPre2014}
{\sc K. Matsumoto and H. Matui},
{\it Continuous orbit equivalence of topological Markov shifts 
and dynamical zeta functions}, preprint, arXiv:1403.0719.




\bibitem{MatuiPLMS}
{\sc H. Matui}, 
{\it Homology and topological full groups of {\'e}tale groupoids on totally disconnected spaces},
 Proc. London Math. Soc. {\bf 104}(2012), 
 pp.\ 27--56.

\bibitem{MatuiPre2012}
{\sc H. Matui}, 
{\it Topologicasl full groups of one-sided shifts of finite type},
preprint, arXiv:1210.5800, to appear in J. Reine Angew. Math..



\bibitem{PS}
{\sc W. Parry and D. Sullivan},
{\it A topological invariant for flows on one-dimensional spaces},
Topology
 {\bf 14}(1975), pp.\ 297--299.




\bibitem{Po}
{\sc Y. T. Poon},
{\it A K-theoretic invariant for dynamical systems}, 
Trans.\  Amer.\ Math.\ Soc.\ 
{\bf 311}(1989), pp.\ 513--533.









\bibitem{Renault}{\sc J. Renault},
{\it A groupod approach to $C^*$-algebras},
Lecture Notes in Math.  793, 
Springer-Verlag, Berlin, Heidelberg and New York (1980).

































\bibitem{Wi} {\sc R. F. Williams},
{\it Classification of subshifts of finite type}, 
 Ann.\ Math.\  {\bf 98}(1973), pp.\ 120--153.
  erratum, Ann.\ Math.\
{\bf 99}(1974), pp.\ 380-381.

\end{thebibliography}
\end{document}